\documentclass[12pt]{article}
\usepackage{graphicx}
\usepackage{setspace}
\usepackage{amssymb}
\usepackage{amsmath}
\usepackage{bbm} 
\usepackage{bm} 
\usepackage{subfigure}
\usepackage{cite} 
\usepackage{hyperref}

\usepackage{type1cm}        
\usepackage{makeidx}         
\usepackage{graphicx}        
\usepackage{multicol}        
\usepackage[bottom]{footmisc}

\usepackage{newtxtext}       %
\usepackage[varvw]{newtxmath}       

\usepackage{bbm}
\usepackage{color}
\usepackage[all]{xy}

\makeindex             

\newtheorem{theorem}{Theorem}
\newtheorem{proposition}[theorem]{Proposition}
\newtheorem{conjecture}[theorem]{Conjecture}
\newtheorem{lemma}[theorem]{Lemma}

\newenvironment{proof}[1][Proof]{\textbf{#1.} }{\ \rule{0.5em}{0.5em}}

\begin{document}

\title{Sampling low-spectrum signals on graphs via cluster-concentrated modes: examples.}
\author{Jeffrey A. Hogan\\
School of Mathematical and Physical Sciences\\
University of Newcastle\\
Callaghan NSW 2308 Australia\\
email: {\tt jeff.hogan@newcastle.edu.au} \and
Joseph D. Lakey\\
Department of Mathematical Sciences\\
New Mexico State University\\
Las Cruces, NM 88003--8001\\
email: {\tt jlakey@nmsu.edu}}
\abstract{We establish frame inequalities for signals in Paley--Wiener spaces on two specific families of graphs consisting of combinations of cubes and cycles. The frame elements are localizations to cubes, regarded as clusters in the graphs, of vertex functions that are eigenvectors of certain spatio--spectral limiting operators on graph signals.}
\smallskip\noindent
{\bf keywords:\,}{Graph product, Dirichlet eigenvector, 
 graph Fourier transform
}

\maketitle

\maketitle


\section{Introduction}
 We investigate here,
  in the context of signals on graphs,
 extensions of two specific sampling methods that were mentioned in \cite{higgins_cardinal_1985}, namely
 generalized sampling of bandlimited signals (loc. cit. p.~68)
cf., \cite{papoulis-gensamp,papoulis_1968,brown_1981},
and time and band limiting (loc. cit. p. 65), also known as spatio--spectral limiting in geometric contexts.  The former is now well established in the context of sampling
spectrum-limited signals on graphs while the latter has also been considered in relation to sampling on graphs \cite{tsitsvero_etal_2015}, cf., \cite{Hogan2018,hogan2018spatiospectral}.

We study inner products with eigenvectors of spatio--spectral limiting as generalized sampling measurements for two very specific families of finite graphs
 involving combinations of cubes and cycles, in which the cubes are regarded as \emph{clusters}, 
 and the cycle is regarded as a skeleton that connects the clusters. 
Of particular interest is how sampling inequalities---frame inequalities of the form 
$A\|f\|^2\leq \sum |\langle f,\,\psi_j\rangle|^2\leq B\|f\|^2$ where,  for fixed $\{\psi_j\}$, $\langle f,\,\psi_j\rangle$
are regarded as measurements or generalized samples of vertex functions $f$ in the span of low--spectrum eigenvalues of the 
graph Laplacian---depend on how the cubes and cycles are combined.  This study is motivated by recent work of I. and M. Pesenson \cite{pesenson_2021} who established sampling inequalities on graphs involving a single sampling measurement per cluster. Our study seeks corresponding inequalities involving a number of measurements per cluster according to the number of spectral modes that can be concentrated in each cluster. The reason to study the specific families of graphs examined here is that these graphs
allow an analytic apparatus to estimate or count the  number of such modes explicitly.
Here is an outline of the contents. 

In Sect.  \ref{sec:background} we  provide background,
briefly reviewing generalized sampling in Paley--Wiener spaces on $\mathbb{R}$ in Sect.~\ref{ssec:gensamp},
 and time and band limiting and the role of prolate functions in generalized sampling in Sect.~\ref{ssec:tbl}.
Section \ref{ssec:pesenson} reviews basic concepts of spectral graph theory and a  recent result of  I. and M. Pesenson concerning cluster-wise sampling on graphs.  We review particular properties of cubes and cycle graphs in  Sect.~\ref{ssec:cycle_cube}.

Section \ref{sec:vertexsub} considers certain graphs that we denote by $\mathcal{B}_N\vdash \mathcal{C}_m$ 
 in which a cube $\mathcal{B}_N$ is substituted for each vertex on a cycle $ \mathcal{C}_m$.  Each inserted cube is regarded as a cluster
and 
is connected to two other cubes, each by a single
edge, to form a cycle of clusters.
In Sect.~\ref{ssec:vertexsub}, eigenvectors (of the graph Laplacians) of these graphs are identified as having either \emph{Dirichlet type}, supported on a single copy of $\mathcal{B}_N$, or \emph{Neumann-type}. The latter are eigenvectors of a certain 
\emph{augmented Laplacian} on $\mathcal{B}_N$ modulated by eigenvectors of the cycle.   The structure of the eigenvectors of this graph
is summarized in Thm.~\ref{thm:rtimes_eigenvectors}.
Section \ref{ssec:ssl_vertexsub} discusses a certain spatial-limiting operator $Q$ that truncates a vertex function $f$ to
a single $\mathcal{B}_N$-cluster and a spectral-limiting operator $P$ that projects onto the span of the low-eigenvalue eigenvectors of the Laplacian on  $\mathcal{B}_N\vdash \mathcal{C}_m$. 
Conjecture \ref{conj:pq-eigenvectors} describes the eigenvectors of $PQ$ and proposes a complete set of eigenvectors of 
shifted $PQ$-operators whose eigenvalues should exceed $1/2$, meaning that the corresponding low-spectrum 
eigenvectors are concentrated in a single copy of $\mathcal{B}_N$.
 Section \ref{ssec:specific_rtimes} provides a specific example fixing  the cube parameter $N=7$ and cycle parameter $m=21$. 
 
 Section \ref{sec:cartesian} considers Cartesian products $\mathcal{B}_N\square\,\, \mathcal{C}_m$ of cubes and cycles.   
 Here, each $\mathcal{B}_N$-slice (fixed cycle coordinate)  can be regarded as a
 cluster but, in  contrast to $\mathcal{B}_N\vdash \mathcal{C}_m$, each vertex in a cluster is now adjacent to a unique corresponding vertex in each of the two adjacent
 clusters.  
 Theorem \ref{thm:bnsquarecm} describes orthonormal bases (ONBs) for decompositions of Paley--Wiener spaces on 
 these Cartesian products in terms of components localized on a single cubic slice,  versus components that are distributed over all cubic slices. Section \ref{ssec:N7m21_cartesian} provides a specific example fixing  the same cube parameters 
 $N=7$ and cycle parameter $m=21$ for the Cartesian case.

  Finally,  Sect.~\ref{sec:conclusion}  discusses how our results for specific graph families  might
  extend to other finite graphs.     The Appendix 
  establishes a condition for generalized sampling in the context of spatial and spectral
 limiting on (standard Cayley graphs of) finite abelian groups, which include $\mathcal{B}_N\square\,\, \mathcal{C}_m$.
 The observation is not central to the main theme here, but it provides a general context for an important aspect of the examples presented here.
 Throughout, notation will be introduced as needed to define quantities that use the notation.

\section{Background\label{sec:background}}

\subsection{Generalized sampling \label{ssec:gensamp}}
Generalized sampling  on $\mathbb{R}$ in the sense of Papoulis \cite{papoulis-gensamp}  refers to sampling formulas of the form
\begin{equation}\label{eq:gensamp} f(t)=\sum_{\nu=1}^M \sum_{k=-\infty}^\infty z_\nu(k T) p_\nu(t-k T), \quad f\in {\rm PW}_\Omega(\mathbb{R})
\end{equation}
where $\widehat{z}_\nu(\xi)= \widehat{f}(\xi) \widehat{h}_\nu(\xi)$, which implies that
$z_\nu(k T)= f\ast h(k T)=\langle f,\, h(\cdot -k T)\rangle$
 for real filter functions $h_\nu$, $\nu=1,\dots, M$.
The uniform recovery condition developed in Brown 
\cite{brown_1981} is a uniform lower bound on the determinant of the matrix $A_{\nu \ell}(\omega)=\widehat{h}_\nu(\omega+\ell\sigma)$,
$\nu=1,\dots, M$; $\ell=0,\dots, M-1$
on the interval $-\Omega\leq \xi\leq -\Omega+\sigma$ where $\sigma=2\Omega/M$. In this event 
the reconstruction functions $p_\nu(t)$ have the form $p_\nu(t)=\sum_{\ell=0}^{M-1} \int_{-\Omega}^{\sigma-\Omega}b_{\nu \ell}(\xi) \, e^{2\pi i (\xi +\ell \sigma) t}\, d\xi$ where $b_{\nu \ell}$ is the $(\nu,\ell)$th element of the inverse of the matrix $A$.   
A critical part of the
argument is that the invertibility condition allows uniform reconstruction of the exponential $e^{2\pi i t\xi}$
on  $(-\Omega, \sigma-\Omega)$, then on $(-\Omega,\Omega)$,  in the form
\begin{equation} \label{eq:exp_gensamp} e^{2\pi i t\xi}=\sum_{\nu=1}^M \widehat{h}_\nu(\xi) \sum_{n=-\infty}^\infty y_\nu(t-kT) e^{2\pi i kT\xi};\quad
\widehat{y}_\nu(\xi,t)=\sum_{\ell=0}^{M-1} e^{2\pi i k\sigma t} b_{\nu \ell}(\xi) 
\end{equation}
which is a refinement of the original approach of Papoulis.

\subsection{Time and band limiting, prolate functions, and prolate shift frames\label{ssec:tbl}}
A particular case of generalized sampling arises when the filter functions $h_\nu$ above are eigenfunctions of time- and band-limiting operators.
The bandlimiting operator $P_\Omega$ projects onto the Paley--Wiener space ${\rm PW}_\Omega=\{f\in L^2(\mathbb{R}): \widehat{f}(\xi)=0,\,\, |\xi|>\Omega\}$ where $\hat{f}(\xi)=\int_{-\infty}^\infty f(t)\, e^{-2\pi i t\xi}\, dt$ when $f\in L^1\cap L^2(\mathbb{R})$. 
The time-limiting  operator is $(Q_Tf)(t)=f(t)$ if $|t|\leq T$ and $(Q_Tf)(t)=0$ if $|t|>T$. Up to a dilation factor, the 
eigenfunctions of the operator $P_\Omega Q_T$ are the \emph{prolate spheroidal wave functions} (PSWFs) $\psi_\nu$ 
with eigenvalues $1>\lambda_0>\lambda_1>\cdots$, e.g., \cite{hogan_lakey_tbl}.  The eigenvalues of $P_\Omega Q_T$ are the same as the concentrations 
of $L^2$-norm  on $[-T,T]$, that is,
$\|Q_T\psi\|^2/\|\psi\|^2$.  The famous $2\Omega T$ theorem \cite{landau_widom} states that the number of eigenvalues of $P_\Omega Q_{T/2}$ that are \emph{close to one}  is approximately $2\Omega T-O(\log(2\Omega T))$ for large enough values of $\Omega T$.
This observation is tied to the \emph{spectral accumulation property} of prolates,
namely that  $\sum_{\nu=0}^\infty\lambda_\nu |\widehat{\psi}_\nu (\xi)|^2=1$ on $[-\Omega,\Omega]$ (and vanishes for $|\xi|>\Omega$). Another useful property of prolates is that 
they are orthogonal on $[-T,T]$: $\langle Q_T\psi_n,\, \psi_m\rangle =\lambda_n\delta_{n,m}$.

\begin{theorem}\label{thm:prolate_riesz_basis}
Let $\psi_\nu$ denote the $\nu$th prolate bandlimited to $[-\Omega,\, \Omega]$ and time-concentrated in $[-1,1]$.
If $N=2\Omega \alpha\in\mathbb{N}$ then the functions 
$\{\psi_\nu(\cdot -\alpha k)\}_{\nu=0,k\in\mathbb{Z}}^{N-1}$ form a Riesz basis for ${\rm PW}_{\Omega}$. 
That is, there exist constants $0<A\leq B<\infty$
such that for any sequence $\{c_{\nu k}\}_{\nu=0,\, k\in\mathbb{Z}}^{N-1}\in \ell^2( \mathbb{Z}^N)$ one has
\[A\sum_{\nu k} |c_{\nu k}|^2\leq  \|\sum_{\nu k} c_{\nu k}\,\psi_\nu (\cdot-\alpha k)\|^2\leq B\sum_{\nu k} |c_{\nu k}|^2\, .
\]
If $N\geq \lceil 2\Omega \alpha\rceil$ then the functions 
$\{\psi_\nu (\cdot -\alpha k)\}_{\nu=0,k\in\mathbb{Z}}^{N-1}$ form a frame for ${\rm PW}_{\Omega}$: there exist constants $0<A\leq B<\infty$
such that for any $f\in {\rm PW}_{\Omega}$ one has
\[A\|f\|_{L^2}^2\leq  \sum_{\nu=0}^{N-1} \sum_{k\in\mathbb{Z}} |\langle f,\, \psi_\nu(\cdot-\alpha k)\rangle|^2\leq B\|f\|_{L^2}^2\, .
\]
If $N< \lceil 2\Omega \alpha\rceil$ then the lower frame bound fails.
\end{theorem}
The Riesz basis case follows along lines similar to the generalized sampling case as in (\ref{eq:exp_gensamp}),
specifically by verifying the matrix invertibility criterion for generalized sampling.
It is observed further in \cite{hogan_lakey_2015} that the frame or Riesz basis bounds are most snug when $N\approx 4\Omega$,
or $\alpha\approx 2$ so the number of sampling functions used is approximately equal to the time--bandwidth product.
In this event, $\sum_{\nu=0}^{N-1} |\psi_\nu(\xi)|^2$ is smoothly bounded away from zero on $[-\Omega,\,\Omega]$
and the series approximation of $e^{2\pi i t\xi}$ by shifted prolates along the lines of (\ref{eq:exp_gensamp})  is numerically effective.

\subsection{Generalized sampling in Paley--Wiener spaces on graphs and clusters\label{ssec:pesenson}}

 \paragraph{Graphs and their Fourier transforms}
 Let $\mathcal{G}$ be a simple unweighted, undirected connected finite graph with $n$ vertices $V(\mathcal{G})$ and edges $E(\mathcal{G})\subset V(\mathcal{G})\times V(\mathcal{G})$. Given an ordering 
 $v_1,\dots, v_{n}$ of $V(\mathcal{G})$, the adjacency matrix $A$ has $(i,j)$-entry equal to one if $(v_i,v_j)\in E(\mathcal{G})$ and equal to zero otherwise.   We write $v\sim w$ if $(v,w)\in E(\mathcal{G})$. 
 When $\mathcal{G}$ is fixed we will simply write $V$ or $E$ for $V(\mathcal{G})$ and $E(\mathcal{G})$ (and similarly for other quantities that refer to the graph $\mathcal{G}$). The space of functions $f:V(\mathcal{G})\to\mathbb{C}$ is denoted $\ell^2(\mathcal{G})$.
 When referring to a graph Laplacian on $\mathcal{G}$ we always mean the unnormalized 
 Laplacian operator $L(\mathcal{G})$ 
 that maps a vertex function $f\in\ell^2(\mathcal{G})$ to
 $(Lf)(v)=\sum_{w\sim v} [f(v)-f(w)]$. 
  If $D$ is the diagonal degree matrix with $i$th diagonal entry $D_{ii}=\#\{w: (v_i,w)\in E\}$
(called the degree ${\rm deg}\, v_i$ of $v_i$) then $L=D-A$.  It is standard to refer to the eigenvalues of $L(\mathcal{G})$ as the eigenvalues or \emph{spectrum} of $\mathcal{G}$.  The operator $L$ is nonnegative and symmetric. It has a single eigenvalue zero (with constant eigenvector) and norm at most $\max\{{\rm deg}\, v\}_{v\in V}$.  We may write the eigenvalues of $\mathcal{G}$ as $0=\lambda_0(\mathcal{G})<\lambda_1(\mathcal{G})\leq \cdots\leq \lambda_n(\mathcal{G})$.
The graph Fourier transform $F(\mathcal{G})$ is represented by the unitary matrix whose columns are the corresponding unit-normalized eigenvectors of $L$.   We assume that the columns are ordered by increasing Laplacian eigenvalue.
Given $\Omega>0$ the space ${\rm PW}_\Omega$ is defined as the span of the Laplacian eigenvectors having eigenvalue
at most $\Omega$.  If $F_\Omega$ is the matrix consisting of the columns of $F(\mathcal{G})$ with Laplacian eigenvalue at most $\Omega$ then the orthogonal projection onto ${\rm PW}_\Omega$ 
is expressed by the matrix $F_\Omega F_\Omega^\ast$.

 \paragraph{A cluster sampling inequality of I. and M. Pesenson}
In \cite{pesenson_2021},  I.~and M.~Pesenson consider the problem of sampling and interpolation of low-spectrum information associated with 
clusters in undirected weighted graphs.  Some notation is needed to describe a version of the result of  \cite{pesenson_2021} of greatest relevance here.
Let $\mathcal{G}$ be a connected finite unweighted graph and $\mathcal{S}=\{S_j\}_{j=1}^m$
a partition of $V(\mathcal{G})$. Let $L_j$ be the  Laplace operator of the induced graph $\mathcal{G}_j$ on $S_j$ with first nonzero eigenvalue $\lambda_{1,j}$ and $\varphi_{0,j}$ its zero eigenfunction. Assume for each $j$ a \emph{sampling function} $\psi_j\in\ell^2(\mathcal{G})$ is supported in $S_j$ 
and satisfies $\langle \psi_j,\, \varphi_{0,j})\neq 0$.  We assume that $\|\psi_j\|=1$ and set $\theta_j=1/{|\langle \varphi_{0,j},\psi_j\rangle|^2}$, which is generally larger than one (and equal to one precisely when $\psi_j$ is equal to $\varphi_{0,j}$).
The Pesensons set $\Lambda=\inf_j \lambda_{1,j}$ and $\Theta=\sup_{j} \theta_j$ 
when the partition $\{S_j\}$ and  sampling vectors $\{\psi_j\}$ are fixed.
As a consequence of certain Poincar\'e inequalities, I. and M. Pesenson proved a result that
we state as follows for the case considered here. 

\begin{theorem}\label{thm:pesenson2} For every $f\in {\rm PW}_\Omega$, where $0<\Omega<\Lambda/\Theta$,
the Plancherel--Polya inequalities
\begin{equation}\label{eq:graph_pp} 
\frac{(1-\mu)\epsilon}{(1+\epsilon) \Theta}\|f\|^2 \leq \sum_{j=1}^m |\langle f,\psi_j\rangle|^2\leq C\|f\|^2
\end{equation}
hold for a fixed $C>0$ provided $\mu=(1+\epsilon)\frac{\Theta}{\Lambda}\Omega <1$.
\end{theorem}
It is also proved that the coefficients $\langle f,\psi_j\rangle$ identify $f\in {\rm PW}_\Omega$ uniquely.

The condition on $\Omega $ which implies that $\Omega<\Lambda/\Theta=\inf \lambda_{1,j}\inf |\langle \psi_j,\,\varphi_{0,j}\rangle|^2$ is fairly restrictive since it implies that $\Omega<\lambda_{1,j}$. 
The Pesensons raise the question  how big is the admissible interval $[0,\Lambda/\Theta)$.
In the limiting case of a disconnected graph, this interval contains only the zero eigenvalue of each cluster.
In the weighted case, if one allows for  very weak weighted connections
between disjoint clusters, and accounts for continuity of the spectrum with respect to the weights, the admissible interval will then contain a number of eigenvalues equal to the number of clusters.

The Pesensons go on to give examples of sampling  functionals defined by $\psi_j$, including cluster averages,  
pointwise samples, and variational splines, explaining how the constants in (\ref{eq:graph_pp}) depend on these choices.
They then raise the question whether similar results can be obtained for a segment of the spectrum strictly
larger than  the interval $[0,\Lambda/\Theta)$. 
This
raises a broader question, whether sampling
inequalities like (\ref{eq:graph_pp}) for clustered graphs can be extended beyond Paley--Wiener spaces that encode
 a single eigenvalue per cluster.  
Stated differently:
for a given graph $\mathcal{G}$ and vertex partition $\mathcal{S}=\{S_j\}$, 
is there a
frame for ${\rm PW}_\Omega$ in which each frame element is well concentrated in one of the $S_j$?
 In contrast to the case of a homogeneous setting like the real line in Thm.~\ref{thm:prolate_riesz_basis},
 how  sampling or frame inequalities apply will depend on the geometry of the clusters relative to the whole graph.
We will consider two similar looking, but structurally different examples in Sects.~\ref{sec:vertexsub} and \ref{sec:cartesian} to elaborate on this point.
First,  we make some general remarks about how local versus global geometry is reflected in the Laplacian.

Consistent with classical sampling theorems, in order to obtain sampling bounds for larger values of $\Omega$, including
values that account for larger eigenvalues of the cluster Laplacians, requires enough sampling
functionals per cluster to account for larger cluster eigenmodes.  The unnormalized graph Laplacian $L$ can be expressed as $L=\tilde{L}+\sum_j L_j$
where $\tilde{L}$ accounts for connections between clusters. If $v$ is an \emph{insulated} vertex, meaning that
all neighbors of $v$ in $G$ lie in $S_j$, then $(Lf)(v)=(L_jf)(v)$.  
More generally, if $\tilde{L}$ has small norm, it suggests that  an inequality like (\ref{eq:graph_pp})
should hold on a subspace of ${\rm PW}_\Omega$ accounted by eigenvalues of separate $L_j$ smaller than 
$\Omega -\|\tilde{L}\|$, leaving the question how to account for global behavior encoded in $\tilde{L}$.
 
 \subsection{Cycles and cubes 
 \label{ssec:cycle_cube}}
 In Sects.~\ref{sec:vertexsub} and \ref{sec:cartesian} we study examples of graphs amenable to analytical methods to study concentration of low-spectrum components on relatively densely connected subsets of vertices. These graphs are combinations of cycles and cubes whose properties we outline now.

  The $m$-cycle $\mathcal{C}_m$ is the graph whose vertices are the elements of $\mathbb{Z}_m$ (we denote simply by $\ell$ 
  the vertex of the element $\ell\in\mathbb{Z}_m$)
 and adjacency is defined by $\ell\sim \ell'$ if and only if $\ell'= \ell\pm 1\mod m$. 
  The Laplacian $L(\mathcal{C}_m)$ of $\mathcal{C}_m$ is usually represented by the $m\times m$ matrix
 $2I-A_m$ where $A_m$ is the symmetric matrix with entries equal to $A_m (k,k+1)=1=A_m(k-1,k)$ and $A_m(k,\ell)=0$ if $|(k-\ell)\mod m|\neq 1$ (here entries are indexed by $0,\dots, m-1$ with addition modulo $m$).
 The eigenvectors of $L(\mathcal{C}_m)$ are, up to normalization,  $E_k(\ell)=e^{2\pi i k\ell/m}/\sqrt{m}$ with eigenvalue $2-2\cos (2\pi  k/m)=4\sin^2(\pi k/m)$.
 When $m=2$ these eigenvalues are zero and two.
 The  Fourier transform of $\mathcal{C}_m$ is represented by the matrix with columns $E_k$. When $m=2$ 
 it is the Haar matrix $H=\frac{1}{\sqrt{2} }\Bigl(\begin{matrix}1&1\\1& -1\end{matrix}\Bigr)$.
 
  The Cartesian product of two graphs $\mathcal{G}$ and $\mathcal{H}$ is the graph $\mathcal{G}\,\square\, \mathcal{H}$
 with $V(\mathcal{G}\,\square\, \mathcal{H})=V(\mathcal{G})\times V( \mathcal{H})$ and 
 $E(\mathcal{G}\,\square\, \mathcal{H})$ defined by 
 $(u_1,v_1)\sim (u_2,v_2)$ if $u_1\sim u_2$ and $v_1=v_2$ or $u_1=u_2$ and $v_1\sim v_2$.
The Laplacian $L(\mathcal{G}\,\square\, \mathcal{H})$ is the Kronecker product 
$L(\mathcal{G})\otimes L(\mathcal{H})$ of $L(\mathcal{G})$ and $L(\mathcal{H})$ 
(if $A$ is $m\times n$ then $A\otimes B$ has $m\times n$ blocks, and $(k,\ell)$-th block equal to $a_{k\ell} B$)  
and therefore the eigenvalues of $L(\mathcal{G}\,\square \,\mathcal{H})$ are sums of eigenvalues of 
 $L(\mathcal{G})$ and $L(\mathcal{H})$ (with corresponding multiplicities).

 The Boolean cube $\mathcal{B}_N$ is the $N$-fold Cartesian product of $\mathcal{C}_2$. 
It has vertices $\mathbb{Z}_2^N$ where two vertices $v=(\epsilon_1,\dots, \epsilon_N)$ and 
 $w=(\epsilon_1',\dots, \epsilon_N')$ are adjacent precisely when $v-w=e_i$ (subtraction in $\mathbb{Z}_2^N$) for some $i\in \{1,\dots, N\}$ where $e_i$ has one in the $i$th coordinate
 and zeros in the other coordinates. The unnormalized Laplacian eigenvalues are $0, 2,\dots, 2N$ where $2k$ has multiplicity $\binom{N}{k}$.
 The Fourier transform of $\mathcal{B}_N$
 is a Hadamard matrix $H_N$ of size $2^N\times 2^N$ whose columns, up to redindexing (reordering of rows and columns) are those of
 the $N$th Kronecker product of the Haar matrix $H$.
 The Fourier transform of $\mathbb{Z}_2^N$ can also be indexed by elements $\mathbf{\gamma}\in\mathbb{Z}_2^N$.
 In this case, if $h_\gamma$ is the column  of $H_N$ corresponding to $\gamma\in\mathbb{Z}_2^N$ then $h_\gamma(v)=2^{-N/2} (-1)^{\langle v,\gamma\rangle}$.
Further elaboration on these basic concepts on cubes relevant to following discussion can be found in \cite{Hogan2018,hogan2018spatiospectral}

 \section{Example 1: Sampling on Vertex substitution of cubes on cycles. \label{sec:vertexsub}}
 
  \subsection{Vertex substitution of cubes on cycles  \label{ssec:vertexsub}}
 Spectral properties of edge substitution graphs were studied by Strichartz in  \cite{strichartz_2010} and in \cite{strichartz_2016},
where sampling properties were also established.   
In edge substitution, a copy of a fixed graph $\mathcal{H}$ which has two designated vertices $v_0$ and $v_1$, is placed along each edge of
a graph $\mathcal{G}$ in such a way that $v_0$ and $v_1$ are identified with the terminals of each corresponding edge.
In the case of a $2$-regular graph (a cycle) it makes sense, instead, 
to substitute a fixed graph for each vertex of the cycle in such a way that  $v_0$ and $v_1$ become terminals of successive edges.  
We denote such a substitution of $\mathcal{H}$ for each vertex of $\mathcal{C}_m$ by 
$\mathcal{H}\vdash \mathcal{C}_m$. The symbol ``$\vdash$'' 
is used to indicate the asymmetric roles of the component graphs. 

 Our first set of examples of structured graphs consisting of several connected clusters are the graphs $\mathcal{B}_N\vdash \mathcal{C}_m$, in which each vertex of a cycle is replaced by a copy of $\mathcal{B}_N$.  We will denote by
  $\mathcal{B}_N^k$ the copy of $\mathcal{B}_N$ placed on the $k$th vertex of $\mathcal{C}_m$ ($k=0,\dots, m-1$)
  and by  $v_{\mathbf{0}}^k$ and $v_{\mathbf{1}}^k$ the vertices of $\mathcal{B}_N^k$ corresponding to the zero element
  $\mathbf{0}=(0,\dots,0)$ and ones-element $\mathbf{1}=(1,\dots, 1)$ of $\mathbb{Z}_2^N$.
 Since the transformation that replaces $v=(\epsilon_1,\dots, \epsilon_N)\in\mathbb{Z}_2^N$ by 
$\tilde{v}=(1+\epsilon_1,\dots, 1+\epsilon_N)\in\mathbb{Z}_2^N$  is an isomorphism of $\mathcal{B}_N$, it is not really critical whether we assign $v_{\mathbf{0}}^k$ or $v_{\mathbf{1}}^k$ as the neighbor of $v_{\mathbf{1}}^{k-1}$  (resp.~$v_{\mathbf{0}}^{k+1}$), but for simplicity we always assign $v_{\mathbf{1}}^{k-1} \sim v_{\mathbf{0}}^{k}$ and 
$v_{\mathbf{1}}^{k} \sim v_{\mathbf{0}}^{k+1}$, see Fig.~\ref{fig:b2c3}. Generically we will denote by $v^k$ the
copy of $v\in V(\mathcal{B}_N)$ lying in $ V(\mathcal{B}_N^k)$. 
We refer to $\mathcal{B}_N\vdash \mathcal{C}_m$ as the \emph{vertex substitution} of $\mathcal{B}_N$ on $\mathcal{C}_m$.   As in \cite{strichartz_2010}, we will use the term \emph{Dirichlet eigenvector} to refer to an eigenvector  of the 
Laplacian $L(\mathcal{B}_N)$ that vanishes at ${\mathbf{0}}$ and ${\mathbf{1}}$.

Denote by 
$\iota_k$ the mapping of a vertex function $\varphi$ on $\mathcal{B}_N$ to one on $\mathcal{B}_N\vdash \mathcal{C}_m$
such that $(\iota_k \varphi)(u^k)=\varphi(u)$, $u\in V(\mathcal{B}_N)$ and $(\iota_k \varphi)(u^\ell)=0$, $\ell\neq k$.
Any vertex in $V(\mathcal{B}_N\vdash \mathcal{C}_m)$ that is not a copy of $v_{\mathbf{0}}$ or $v_{\mathbf{1}}$ will be called
an \emph{insulated vertex}.  It is clear that, at any insulated vertex $u^k\in V(\mathcal{B}_N^k)$,
for any vertex function $f\in \ell^2(\mathcal{B}_N)$,  
$(L(\mathcal{B}_N\vdash \mathcal{C}_m))(\iota_k f)(u^k)= (\iota_k\, L(\mathcal{B}_N)f)(u^k)$.
The following is then clear.

\begin{lemma}\label{lem:dirichlet} If $\varphi$ is a Dirichlet eigenvector of $L(\mathcal{B}_N)$ then $\iota_k(\varphi)$ is  an eigenvector of
$L(\mathcal{B}_N\vdash \mathcal{C}_m)$ with the same eigenvalue as the $L(\mathcal{B}_N)$-eigenvalue of $\varphi$,
and vanishing outside  $\mathcal{B}_N^k$.
\end{lemma}
We will also refer to any function of the form $\iota_k\varphi$, where $\varphi$ is a Dirichlet eigenvector of  $L(\mathcal{B}_N)$,
as a Dirichlet eigenvector (of $L(\mathcal{B}_N\vdash \mathcal{C}_m)$).

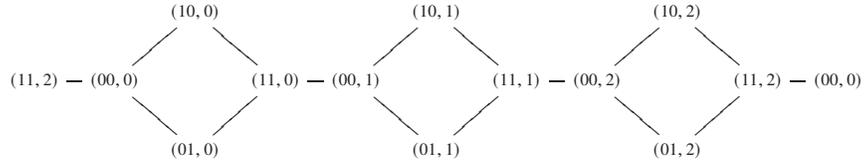
\begin{figure}[tbhp]
\begin{centering}
\vspace{1cm}
\begin{tiny}
\xymatrixrowsep{0.5cm}
\xymatrixcolsep{0.2cm}
$\xymatrix{
&& (10,0) \ar@{-}[rd]& & & (10,1) \ar@{-}[rd]& && (10,2) \ar@{-}[rd]&&\\
(11,2)\ar@{-}[r]&(00,0) \ar@{-}[ru]   \ar@{-}[rd]  & & (11,0) \ar@{-}[r] & (00,1) \ar@{-}[ru]   \ar@{-}[rd]  & & (11,1) \ar@{-}[r]  &(00,2) \ar@{-}[ru]   \ar@{-}[rd]  & & (11,2)\ar@{-}[r]& (00,0)\\
& & (01,0) \ar@{-}[ru]& & & (01,1) \ar@{-}[ru]& && (01,2) \ar@{-}[ru]&&\\
 }$
\end{tiny}
\caption{\label{fig:b2c3} Graphic representation of $\mathcal{B}_2\vdash\mathcal{C}_3$}
\end{centering}
\end{figure}

  \begin{lemma} \label{lem:dirichlet} For each fixed $K$ ($0<K<N$) the span of the $2K$-Dirichlet eigenvectors
  on $\mathcal{B}_N\vdash \mathcal{C}_m$ has dimension $m\bigl(\binom{N}{K}-1\big)$.
  \end{lemma}
  
  \begin{proof} There are $\binom{N}{K}$ linearly independent (mutually orthogonal) Hadamard eigenvectors $h_\gamma$ with 
  $L(\mathcal{B}_N)$-eigenvalue $2K$.  
 Recall that $h_\gamma(v)=2^{-N/2} (-1)^{\langle v,\gamma\rangle}$. In particular,  $h_\gamma({\mathbf{0}})=2^{-N/2}$
 for every $\gamma$ and  $h_\gamma({\mathbf{1}})=2^{-N/2}(-1)^{\#\{j:\gamma_j=1\}}$ for each 
 $\gamma\in\mathbb{Z}_2^N$.   Therefore $h_\gamma({\mathbf{1}})=2^{-N/2} (-1)^K$ when $\sum \gamma_i=K$ as
 is the case for each $2K$-eigenvector of $L(\mathcal{B}_N)$.
An element of the span of the $2K$-eigenvectors has the form $\sum_{|\gamma|=K} c_\gamma h_\gamma$.
  The 
  closed subspace generated by coefficients $\{\{c_\gamma\}_{|\gamma|=K}:\sum c_\gamma=0\}$ has one-dimensional orthocomplement ($c_\gamma$ constant)
and therefore has dimension $\binom{N}{K}-1$.  
This subspace of the $2K$-eigenspace of $L(\mathcal{B}_N)$ evidently contains the Dirichlet eigenvectors of $\mathcal{B}_N$.
There is one isomorphic copy of this subspace associated with block $\mathcal{B}_N^k$  in 
$\ell^2(\mathcal{B}_N\vdash \mathcal{C}_m)$. Since these copies have disjoint supports, the dimension of the direct sum of these copies is $m\bigl(\binom{N}{K}-1\bigr)$.
  \end{proof}
  
  This raises the question of eigenvalues of $L(\mathcal{B}_N\vdash \mathcal{C}_m)$ that are not Dirichlet.
It is clear from the proof of Lem.~\ref{lem:dirichlet} that $L(\mathcal{B}_N)$ itself has $\binom{N}{K}-1$ Dirichlet
  eigenvalues equal to $2K$ ($0<K<N$), leaving one eigenvalue for each $K$ that is not Dirichlet. 
  We claim that the remaining $2K$-eigenvalue of $L(\mathcal{B}_N)$ has an eigenvector of the form  
  $\varphi=c_K\sum_{|\gamma|=K} h_\gamma$.
  We refer to such $\varphi$ as the Neumann eigenvector of $L(\mathcal{B}_N)$ with eigenvalue $2K$.

   \begin{lemma} \label{lem:neumann} For each fixed $K$ ($0\leq K\leq N$) the vector $h_{n,K}=c_K\sum_{|\gamma|=K} h_\gamma$ 
   ($c_K$ is chosen to make $\|h_{n,K}\|=1$)
   is a Neumann eigenvector of $L(\mathcal{B}_N)$. 
   The vectors $h_{n,K}$ are mutually orthogonal for different $K$ and are orthogonal to the Dirichlet eigenvectors of $L(\mathcal{B}_N)$.
  \end{lemma}
  
  \begin{proof} Since each $h_\gamma$ with $|\gamma|=K$ is a $2K$-eigenvector of $L(\mathcal{B}_N)$, the same
  holds for  $\sum_{|\gamma|=K} h_\gamma$.  Since the Hadamard vectors are orthogonal ($\langle h_\gamma,\, h_{\gamma'}\rangle=\delta_{\gamma,\gamma'}$), it follows that $\langle h_{n,K},\, h_{n,K'}\rangle=0$ if $K\neq K'$.  Orthogonality between
  Dirichlet and Neumann eigenvectors for different $K$ follows in the same way. For the same $K$, if $\varphi=\sum_{|\gamma|=K} c_\gamma h_\gamma$ is such that $\sum_{|\gamma|=K} c_\gamma=0$ then 
  \[\langle \varphi,\, h_{n,K}\rangle 
  =c_K\sum_{|\gamma|=|\gamma'|=K} c_\gamma \langle h_\gamma, \, h_\gamma'\rangle = 
 c_K\sum_{|\gamma|=|\gamma'|=K} c_\gamma\delta_{\gamma,\gamma'}=   c_K\sum_{|\gamma|=K} c_\gamma  = 0 \, .
  \]
  This proves that $h_{n,K}$ is orthogonal to any Dirichlet eigenvector and completes the proof of the Lemma.
  \end{proof}

Concatenating $2K$-Dirichlet eigenvectors supported on $\mathcal{B}_N^k$, $k=0,\dots, m-1$
  results in a $2K$-eigenvector of $L(\mathcal{B}_N\vdash \mathcal{C}_m)$.  Remaining eigenvectors of 
  $L(\mathcal{B}_N\vdash \mathcal{C}_m)$ are formed by joining elements of Neumann eigenspaces on the
  $\mathcal{B}_N^k$ in such a way that the resulting vertex function on $\mathcal{B}_N\vdash \mathcal{C}_m$ is 
  a global eigenvector of $L(\mathcal{B}_N\vdash \mathcal{C}_m)$.  We do this by leveraging the exponential structure of 
  eigenvectors of $L(\mathcal{C}_m)$ to define \emph{augmented Laplacian} operators on  $\mathcal{B}_N$
  whose eigenvectors can be concatenated to form eigenvectors of $L(\mathcal{B}_N\vdash \mathcal{C}_m)$.
  
  For $\alpha\in \mathbb{C}\setminus 0$, define a \emph{corner matrix} $C_\alpha$ of size $2^N\times 2^N$ such that
  \[C_\alpha(1,1)=1=C_\alpha(2^N,2^N),\,  C_\alpha(1,2^N)=-\frac{1}{\alpha}, \,\, {\rm and }\,\, C_\alpha(2^N,1)=-\alpha \]
  and $C_\alpha(i,j)=0$ otherwise. 
  $C_\alpha$ has a single eigenvalue equal to two and all other eigenvalues equal to zero.
  Recall that we use lexicographic ordering on $\mathcal{B}_N$ so that for a $2^N\times 2^N$  matrix $A$ acting on $\ell^2(\mathcal{B}_N)$
  and vertex function $f$,  the first column of $Af$ is $f(\mathbf{0})$ times the first column of $A$ and the last column of 
  $Af$ is $f(\mathbf{1})$ times the last column of $A$.  Therefore, $C_\alpha f$ is the vector whose first entry is
   $f(\mathbf{0})-f(\mathbf{1})/\alpha$, whose last entry is $f(\mathbf{1})-\alpha f(\mathbf{0})$, and whose other entries
   are equal to zero.  We now assume that $|\alpha|=1$ and define the augmented Laplacian 
   \begin{equation}\label{eq:augmented_laplacian} L_\alpha(\mathcal{B}_N)=L(\mathcal{B}_N)+C_\alpha\, . \end{equation}
   
At any insulated vertex $v\in V(\mathcal{B}_N)$, $(L_\alpha f)(v)=(Lf)(v)$ ($L=L(\mathcal{B}_N)$).
   In particular, any $2K$-Dirichlet eigenvector of $L$ is also a $2K$-eigenvector of $L_\alpha$.
On the other hand, $(L_\alpha f)(\mathbf{0})=(Lf)(\mathbf{0})+f(\mathbf{0})-f(\mathbf{1})/\alpha$ and 
   $(L_\alpha f)(\mathbf{1})=(Lf)(\mathbf{1})+f(\mathbf{1})-\alpha f(\mathbf{0})$.  
   The non-Dirichlet eigenvectors of $L_\alpha$ can be regarded as perturbations of the Neumann eigenvectors of $L$.
   Since the norm of $C_\alpha$ is two and since the output vectors of $C_\alpha$ are zero in all but the first and last rows, 
   whereas the Neumann eigenvectors are nonzero in all rows, the norm of $C_\alpha$ on the Neumann subspace of $\ell^2(\mathcal{B}_N)$ is strictly smaller than two, and each non-Dirichlet eigenvector of $L_\alpha$ can be regarded
   as a nonnegative perturbation of the Neumann eigenvector of $L$ having the closest eigenvalue. 
  We summarize this analysis of Dirichlet and \emph{Neumann-type} eigenvectors of $L_\alpha(\mathcal{B}_N)$ as follows.
  
     \begin{lemma} \label{lem:alphalaplace} Let $|\alpha|=1$. The operator $L_\alpha(\mathcal{B}_N)$ has a complete set of eigenvectors of the following form: (i) $h$ is a Dirichlet eigenvector of $L(\mathcal{B}_N)$ (with the same eigenvalue), or (ii) $h=\sum_{\kappa=0}^N
   c_\kappa(\alpha) h_{n,\kappa}$. 
   There are $N+1$ mutually orthogonal eigenvectors of 
   $L_\alpha (\mathcal{B}_N)$ of the second type.
  \end{lemma}
  
  \begin{proof}
  The different eigenvectors of type (ii) have eigenvalue $2K+\epsilon_K$ ($\epsilon_K\in (0,2)$)
  for different $K\in\{0,\dots,N\}$ and are therefore mutually orthogonal. 
  \end{proof}

  \paragraph{Augmented Laplacians and eigenvectors of $L(\mathcal{B}_N\vdash \mathcal{C}_m)$}
  The eigenvectors of $L( \mathcal{C}_m)$ are $E_\nu=\{e^{2\pi i \nu\ell/m}/\sqrt{m}\}_{\ell=0}^{m-1}$.   
  For 
  $f\in \ell^2(\mathcal{B}_N\vdash \mathcal{C}_m)$ one has
  \[(L(\mathcal{B}_N\vdash \mathcal{C}_m) f)(v_{\mathbf{0}}^k)=(L(\mathcal{B}_N^k) f)(v_{\mathbf{0}}^k)+[f(v_{\mathbf{0}}^k)-f(v_{\mathbf{1}}^{k-1})]\]
  and similarly
  \[(L(\mathcal{B}_N\vdash \mathcal{C}_m) f)(v_{\mathbf{1}}^k)=(L(\mathcal{B}_N^k) f)(v_{\mathbf{1}}^k)+[ f(v_{\mathbf{1}}^k)- f(v_{\mathbf{0}}^{k+1})],\] 
  while  $(L(\mathcal{B}_N\vdash \mathcal{C}_m) f)(v) =(L(\mathcal{B}_N^k) f)(v)$ for any insulated vertex 
  $v\in V( \mathcal{B}_N^k) $.    Suppose now that the sample vectors $\{f(v_{\mathbf{0}}^k)\}_{k=0}^{m-1}$ and $\{f(v_{\mathbf{1}}^k)\}_{k=0}^{m-1}$ 
form eigenvectors of $L( \mathcal{C}_m)$.  Then there is a value $\nu\in \{0,\dots, m-1\}$ such
  that $f(v_{\mathbf{0}}^{k+1}) =e^{2\pi i \nu/m} f(v_{\mathbf{0}}^{k})$.  Assume that, for the same $\nu$, we also have
 $ f(v_{\mathbf{1}}^{k}) =e^{2\pi i \nu/m} f(v_{\mathbf{1}}^{k-1})$.  In this event,  for $\alpha=e^{2\pi i \nu/m}$ 
 one then has
   \[(L(\mathcal{B}_N\vdash \mathcal{C}_m) f)(v_{\mathbf{0}}^k)=(L(\mathcal{B}_N^k) f)(v_{\mathbf{0}}^k)
   +[f(v_{\mathbf{0}}^k)- e^{-2\pi i \nu/m} f(v_{\mathbf{1}}^{k}) ]=L_\alpha (\mathcal{B}_N^k) f)(v_{\mathbf{0}}^k)\]
  and similarly
  \[(L(\mathcal{B}_N\vdash \mathcal{C}_m) f)(v_{\mathbf{1}}^k)=(L(\mathcal{B}_N^k) f)(v_{\mathbf{1}}^k)
  +[f(v_{\mathbf{1}}^k)- e^{2\pi i \nu/m} f(v_{\mathbf{0}}^{k})] =L_\alpha (\mathcal{B}_N^k) f)(v_{\mathbf{1}}^k).\]
   
  Suppose now that the restriction of $\varphi$ to $\mathcal{B}_N^0$ is an eigenvector of $L_\alpha(\mathcal{B}_N^0)$
  where $\alpha =e^{ 2\pi i \nu/m}$, in other words, this restriction has the form $\iota_0 f$ where $f$   
  is an eigenvector of $L_\alpha(\mathcal{B}_N)$, and suppose that the restriction of $\varphi$ to $\mathcal{B}_N^k$ is equal to $e^{2\pi i \nu k/m} \iota_k f  $ for the same $f$, for each $k=0,\dots, m$.  We conclude from the observations above then that 
  $\varphi$ is an eigenvector of  $L(\mathcal{B}_N\vdash \mathcal{C}_m)$ whose eigenvalue is equal to the 
  $L_\alpha(\mathcal{B}_N)$-eigenvalue of $f$ for $\alpha =e^{ 2\pi i \nu/m}$.
  We can use Lem.~\ref{lem:alphalaplace} to describe a complete family of eigenvectors of $L(\mathcal{B}_N\vdash \mathcal{C}_m)$ as follows.
  
  \begin{theorem}\label{thm:rtimes_eigenvectors}  The eigenvectors $\varphi$ of $L(\mathcal{B}_N\vdash \mathcal{C}_m)$
  have one of the following types:

\noindent (i)  Dirichlet type:  $\varphi$ is  supported in $\mathcal{B}_N^k$ for some $k\in \mathbb{Z}_m$ 
and is a Dirichlet $2K$-eigenvector
  of $L(\mathcal{B}_N^k)$  for a fixed value of $K\in \{0,1,\dots, N\}$, or
 
\noindent  (ii) Neumann type: the sample vectors $\{\varphi(v_{\mathbf{0}}^k)\}$ and $\{\varphi(v_{\mathbf{1}}^k)\}$, $k=0,\dots, m-1$, are eigenvectors of $L(\mathcal{C}_m)$ and the restriction
  of $\varphi$ to  $\mathcal{B}_N^{(k)}$ is an eigenvector of  $L_\alpha(\mathcal{B}_N)$, $\alpha=e^{2\pi i \nu/m}$
  for a fixed $\nu\in\mathbb{Z}_m$. 
  
 \noindent  The number of linearly independent $2K$-eigenvectors of Dirichlet type  is 
  $m\bigl(\binom{N}{K}-1\bigr)$ for each $K=0,\dots, N$ and the number of linearly independent eigenvectors of Neumann type 
  is $m(N+1)$, with values indexed by $K$ and $\nu$. 
  Together, the linearly independent sets of eigenvectors of Dirichlet or Neumann type form a complete set of $m 2^N$ eigenvectors of 
   $L(\mathcal{B}_N\vdash \mathcal{C}_m)$.
  \end{theorem}
  It is worth mentioning that for each group $K=0,\dots, N$ there is a collection of $m$ eigenvectors of Neumann type  
  whose eigenvalues are of the form $2K+\epsilon_{K,\nu}$ where $\epsilon_{K,\nu}\in (0,2)$,
  see Fig.~\ref{fig:slbig_evals_b7c21_1to620}.  These eigenvectors have the form $\sum_{\kappa=0}^N c_\kappa(\alpha) h_{n,\kappa}$
  on $\mathcal{B}_N^k$.
  
  
  \begin{proof}[of Thm.~\ref{thm:rtimes_eigenvectors}]
  By Lem.~\ref{lem:dirichlet}, there are $\binom{N}{K}-1$ linearly independent Dirichlet $2K$-eigenvectors of $L(\mathcal{B}_N)$.
  Since images $\iota_k\varphi$ of such eigenvectors for different $k$ have disjoint supports and since
   $(L(\mathcal{B}_N\vdash \mathcal{C}_m) \iota_k\varphi)=\iota_k( L(\mathcal{B}_N)\varphi )$ for a Dirichlet eigenvector, it follows that there
   are $m\bigl(\binom{N}{K}-1 \bigr)$ linearly independent eigenvectors of Dirichlet type.  
   The total number of linearly independent  eigenvectors of Dirichlet type is $\sum_{K=0}^N m\bigl(\binom{N}{K}-1\bigr)=m(2^N-N-1)$.
   
   The analysis above shows that if $\varphi$ is an eigenvector of $L_\alpha(\mathcal{B}_N)$ of norm one with $\alpha=e^{2\pi i \nu/m}$ 
   then the vector whose values  on $\mathcal{B}_N^k$ are of the form $e^{2\pi i \nu k/m} \iota_k\varphi$  are eigenvectors of 
$L(\mathcal{B}_N\vdash \mathcal{C}_m) $. By Lem.~\ref{lem:alphalaplace}, there are $N+1$ linearly independent Neumann-type 
eigenvectors of 
 $L_\alpha(\mathcal{B}_N)$ for each such $\alpha$ and, when $\varphi_0 ,\dots, \varphi_N$ is a linearly independent set of such vectors,
 the corresponding eigenvectors $e^{2\pi i \nu k/m} \iota_k\varphi_j$ also form a linearly independent set of $m(N+1)$ Neumann-type  elements of 
 $\ell^2(\mathcal{B}_N\vdash \mathcal{C}_m) )$.  Since Dirichlet and Neumann type eigenvectors are linearly independent, 
 adding the total numbers of  eigenvectors of each type one obtains a collection of $m2^N={\rm dim}(\ell^2 (\mathcal{B}_N\vdash \mathcal{C}_m) )$
 eigenvectors. Therefore these eigenvectors are complete in $\ell^2 (\mathcal{B}_N\vdash \mathcal{C}_m) $.    \end{proof}

  \bigskip
  
  \subsection{Band limiting and block limiting on $\mathcal{B}_N\vdash \mathcal{C}_m$ \label{ssec:ssl_vertexsub}}
  
  On $\mathcal{B}_N\vdash \mathcal{C}_m$ let $P_\Omega$ denote the operator that projects $f\in \ell^2(\mathcal{B}_N\vdash \mathcal{C}_m)$ onto the span of the eigenvectors of $L(\mathcal{B}_N\vdash \mathcal{C}_m)$ having eigenvalue at most
  $\Omega$ and let $Q$ denote the operator that truncates  $f\in \ell^2(\mathcal{B}_N\vdash \mathcal{C}_m)$  to 
  $\mathcal{B}_N^0$ by multiplying by zero outside  $\mathcal{B}_N^0$.  Since $P_\Omega $ and $Q$ are orthogonal
  projections, the norm of $P_\Omega Q$ is at most one.  Since any Dirichlet eigenvector of $L(\mathcal{B}_N^0)$
  having eigenvalue at most $\Omega$  is in the range of $P_\Omega$, there are nontrivial unit eigenvectors of $P_\Omega Q$
  when $\Omega\geq 2$.   Contrast this to the case of the real line where eigenvalues of the corresponding iterated project operator
   $P_\Omega(\mathbb{R}) Q(\mathbb{R})$---where $P_\Omega(\mathbb{R})$ bandlimits elements of 
   $L^2(\mathbb{R})$ to $[-\Omega,\Omega]$  and $Q(\mathbb{R})$ truncates elements of $L^2(\mathbb{R})$ to $[-1,1]$---has all eigenvalues strictly smaller than one.
  On the other hand, if $f\in {\rm PW}_\Omega(\mathcal{B}_N\vdash \mathcal{C}_m)$ is in the span of the Neumann-type 
  eigenvectors of $L(\mathcal{B}_N\vdash \mathcal{C}_m)$ identified in Thm.~\ref{thm:rtimes_eigenvectors}, 
  then  $f$ is not supported in $\mathcal{B}_N^0$ so, if it is an eigenvector of $P_\Omega Q$, then
  the eigenvalue is strictly smaller than one.

The analysis in Thm.~\ref{thm:rtimes_eigenvectors}  indicates that the dimension of 
${\rm PW}_{2K}$ is equal to 
\begin{equation}\label{eq:pwdim} m \sum_{\kappa=0}^{K} \left(\binom{N}{\kappa}-1\right)+m\sum_{\kappa=0}^{K-1} 1
=m\left(-1+ \sum_{\kappa=0}^{K} \binom{N}{\kappa}\right) =m({\rm dim}_K -1)
\end{equation}
where ${\rm dim}_K=\sum_{\kappa=0}^K \binom{N}{\kappa}$ is the volume of the $K$-ball in $\mathcal{B}_N$.
The first terms in (\ref{eq:pwdim}) come from Dirichlet eigenvectors 
that are perfectly localized on $\mathcal{B}_N^k$ for each $k=0,\dots, m-1$.  Those localized on $\mathcal{B}_N^0$
are also eigenvectors of $PQ$ having eigenvalue equal to one.  
Thus $PQ$ has  $\sum_{\kappa=0}^{K} \bigl(\binom{N}{\kappa}-1\bigr)$ eigenvalues equal to one. 
The second terms in (\ref{eq:pwdim}) correspond to Neumann-type
eigenvectors whose norms are equally distributed over all blocks  $\mathcal{B}_N^k$.  The example below suggests
that  there should be $K$ Neumann-type eigenvalues of $PQ$ that are \emph{close to one}, see 
Conj.~\ref{conj:pq-eigenvectors}, whose proof
requires showing that there are $K$ Neumann-type eigenvalues of $PQ$  larger than $1/2$. In turn, it requires proving that
the vector consisting of the samples $\{\varphi(v^k)\}$ where $v^k$ is the $\mathcal{B}_N^k$-representative of fixed $v\in\mathcal{B}_M$,
 is nearly cardinal for the corresponding eigenvectors $\varphi$, with the zero-block value equal to at least half the $\ell^2$-norm of the sample vector.

The conjecture implies that when $\Omega=2K\in 2\mathbb{N}$, there exists a basis of eigenvectors $\psi$ of $PQ$ for 
${\rm PW}_{2K}(\mathcal{B}_N\vdash\mathcal{C}_m)$, most of whose eigenvalues are equal to one, and
the rest having eigenvalues close to one.   By associating an operator  $PQ^{(k)}$  to each separate cluster $\mathcal{B}_N^k$
then,  the conjecture implies that there is a basis of ${\rm PW}_{2K}(\mathcal{B}_N\vdash\mathcal{C}_m)$
of vectors that are either perfectly or predominantly localized on a single cluster.  For such a basis one will have inequalities
of the form
\begin{equation}\label{eq:concentration} \mu_K \| Q^{(k)} f\|\leq 
\sum_{j=1}^{{\rm dim}_K-1} |\langle Q^{(k)} f,\,\psi_{j,k}\rangle |^2\leq \| Q^{(k)}f\|^2, \quad f\in {\rm PW}_{2K}
\end{equation}
uniformly in the block parameter $k$, where $\mu_K$ is the minimum eigenvalue of the Neumann-type eigenvectors of $PQ$.

We compare these observations  with Thm.~\ref{thm:pesenson2}, which identifies one measurement per cluster to provide a sampling criterion for ${\rm PW}_\Omega$ 
on a generic graph $\mathcal{G}$, but requiring $\Omega$ to be sufficiently small.  While graphs
$\mathcal{B}_N\vdash\mathcal{C}_m$ form an extremely specific family of graphs, 
 a consequence of (\ref{eq:concentration}) 
  is that on each block 
$\mathcal{B}_N^k$ there is an orthonormal family of \emph{measurement} vectors $\psi_{j,k}$  in ${\rm PW}_{2K}$ that are also orthogonal on $\mathcal{B}_{N}^k$,  such that ${\rm dim}_K-(K+1)$ of the vectors are
fully supported in $\mathcal{B}_N^k$, while the remaining $K$ vectors are concentrated in $\mathcal{B}_N^k$, and such that
most of the $\ell^2$-norm of $f\in {\rm PW}_{2K}$ on $\mathcal{B}_N^k$ is accounted for by the measurements  $\langle Q^{(k)} f,\,\psi_{j,k}\rangle$. 
In the specific case outlined in Sect.~\ref{ssec:specific_rtimes}, the value of $\mu_K$ in (\ref{eq:concentration}) is much larger than $1/2$.

To describe the measurement vectors $\psi_{j,k}$ we introduce the notation $f\dashv g$ for vertex functions $f$ and $g$ defined on $\mathcal{C}_m$ and $\mathcal{B}_N$ respectively, to denote 
the function defined on $\mathcal{B}_N\vdash \mathcal{C}$ that takes the value $f(k)\iota_k g(v)$ at the image of $v$ in $\mathcal{B}_N^k$.
Neumann type eigenvectors 
on $L(\mathcal{B}_N\vdash \mathcal{C}_m)$ 
described in Thm.~\ref{thm:rtimes_eigenvectors} 
have the form 
\begin{equation}\label{eq:neumann_type}
 E_\nu \dashv \sum_{\kappa=0}^N c_\kappa(\alpha,K) h_{n, \kappa}
\equiv E_\nu \dashv {\varphi}_{\nu,K,n}
\end{equation}
where $\{h_{n,\kappa}\}_{\kappa=0}^N$  forms an ONB for the  Neumann eigenvectors of $L(\mathcal{B}_N)$
and $ {\varphi}_{\nu,K,n}= \sum_{\kappa=0}^N c_\kappa(\alpha,K) \varphi_{n, \kappa}$ is an
eigenvector of $L_\alpha(\mathcal{B}_N)$ with $\alpha=e^{2\pi i \nu/m}$ as described in Thm.~\ref{thm:rtimes_eigenvectors}.
The index $K$ in (\ref{eq:neumann_type}) indexes the  $L(\mathcal{B}_N\vdash \mathcal{C}_m)$-eigenvalue
of $E_\nu \dashv {\varphi}_{\nu,K,n}$: as suggested by  Thm.~\ref{thm:rtimes_eigenvectors},
for each $K\in \{0,\dots,N\}$,  and each $\nu$, there is a unique Neumann-type 
 $L(\mathcal{B}_N\vdash \mathcal{C}_m)$-eigenvalue $2K+\epsilon_\nu$ with $\epsilon_\nu\in (0,2)$ corresponding
 to the unique $L_\alpha(\mathcal{B}_N)$-eigenvalue of the form $2K+\epsilon$. 
 Neumann-type eigenvectors of $PQ$ then  have the form 
  \[\psi_{\nu,K,n}=\sum_\nu\sum_{\kappa=0}^{K-1} \beta_{\nu,\kappa}  E_\nu \dashv  {\varphi}_{\nu,\kappa,n}
   = \sum_\nu E_\nu \dashv   \sum_{\kappa=0}^{K-1} \beta_{\nu,\kappa} {\varphi}_{\nu,\kappa,n}
 \equiv \sum_\nu E_\nu \dashv   {{\psi}}_{\nu,K}
  \] 
  where
  $P$
  is the orthogonal projection onto ${\rm PW}_{2K} (\mathcal{B}_N\vdash\mathcal{C}_m)$.
  Observe that in this case the $\mathcal{B}_N\vdash\mathcal{C}_m$ Neumann-type components 
  $ {\varphi}_{\nu,\kappa,n}$ contributing to $\psi_{\nu,K,n}$ 
  themselves lie in ${\rm PW}_{2K} (\mathcal{B}_N\vdash\mathcal{C}_m)$.

  \begin{conjecture}  \label{conj:pq-eigenvectors} Let $0<K<N$ and let $P,Q$ be defined as above.  
There are ${\rm dim}_K-1$ eigenvalues of $PQ$ larger than $1/2$, including ${\rm dim}_K-(K+1)$  eigenvalues equal to one and  another $K$ eigenvalues in $(1/2,1)$.  The cyclic shifts of the corresponding eigenvectors are linearly independent in $\ell^2(\mathcal{B}_N\vdash \mathcal{C}_m)$. Therefore,
these vectors form a basis for ${\rm PW}_{2K}$ which has dimension $m ( {\rm dim}_K-1)$.
\end{conjecture}

That the Neumann-type $PQ$ eigenvectors should have eigenvalues larger than $1/2$ would rely on a sort of
  \emph{uniform cardinality} property of the sequences of $2^N$th samples of $\psi_{\nu,k,n}=
   \sum_\nu E_\nu \dashv  {\psi}_{\nu,K}$, see Fig.~\ref{fig:pqbig_evec_128samples_61626364_b7c21_1to384}
   for the case $N=7, m=21$.

    \begin{figure}[tbhp]
  \footnotesize{
\centering 
\includegraphics[width=\textwidth,height=2.5in]{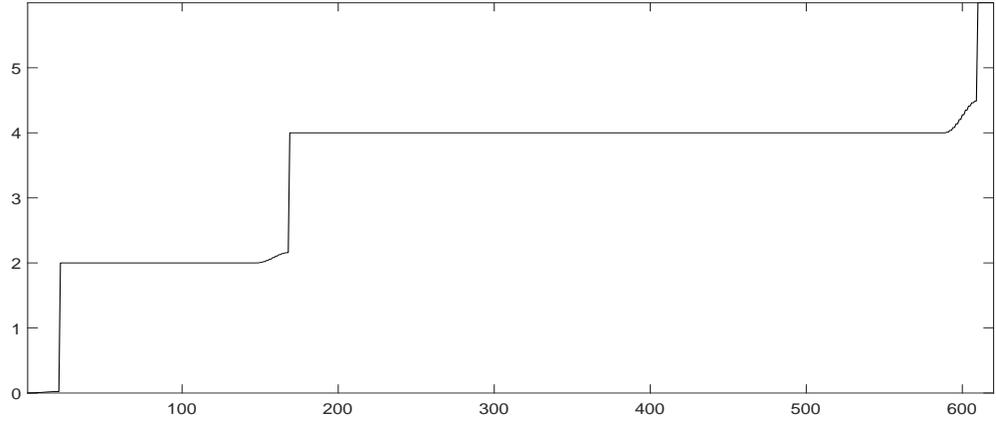}
\caption{\label{fig:slbig_evals_b7c21_1to620}
First 620 eigenvalues of $L(\mathcal{B}_7\vdash \mathcal{C}_{21})$
}
}
\end{figure}

  \subsection{\label{ssec:specific_rtimes} Specific case $N=7$, $K=3$, $m=21$ of $\mathcal{B}_N\vdash \mathcal{C}_m$}

    \begin{figure}[tbhp]
  \footnotesize{
\centering 
\includegraphics[width=\textwidth,height=2.5in]{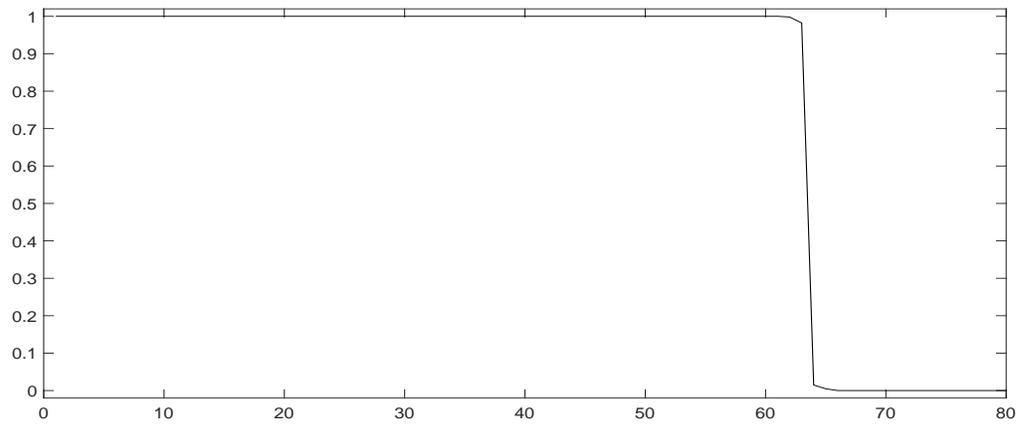}
\caption{\label{fig:spqbig_evals_b7c21_1to80}
Eigenvalues of $PQ$ on $\mathcal{B}_7\vdash \mathcal{C}_{21}$. First 60 eigenvalues equal one, corresponding
to Dirichlet eigenvectors of $L(\mathcal{B}_7)$
}
}
\end{figure}

    \begin{figure}[tbhp]
  \footnotesize{
\centering 
\includegraphics[width=\textwidth,height=2.5in]{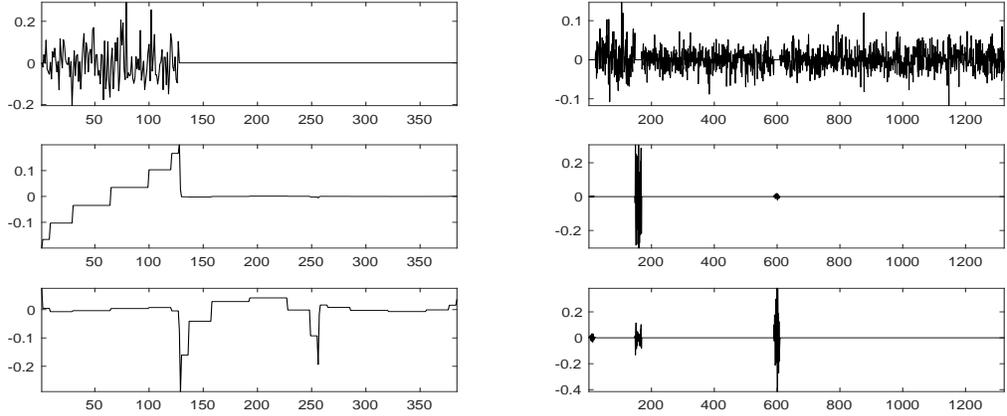}
\caption{\label{fig:pqbig_evecs_and_fourier_32_62_64_b7c21_1to384}
Eigenvectors of $PQ$ on $\mathcal{B}_7\vdash \mathcal{C}_{21}$ corresponding to eigenvalue indices 32, 62 and 64
whose respective $PQ$ eigenvalues are $1$, $0.9982$ and $0.0148$. The first three cycle lengths ($3\cdot 128$) are plotted on the left. Their Fourier
coefficients are plotted on the right.  Eigenvector 32 has its Fourier support in the Dirichlet eigenvalues 
while the other eigenvectors have their Fourier supports in the Neumann-type eigenvalues. 
}
}
\end{figure}

    \begin{figure}[tbhp]
  \footnotesize{
\centering 
\includegraphics[width=\textwidth,height=2.5in]{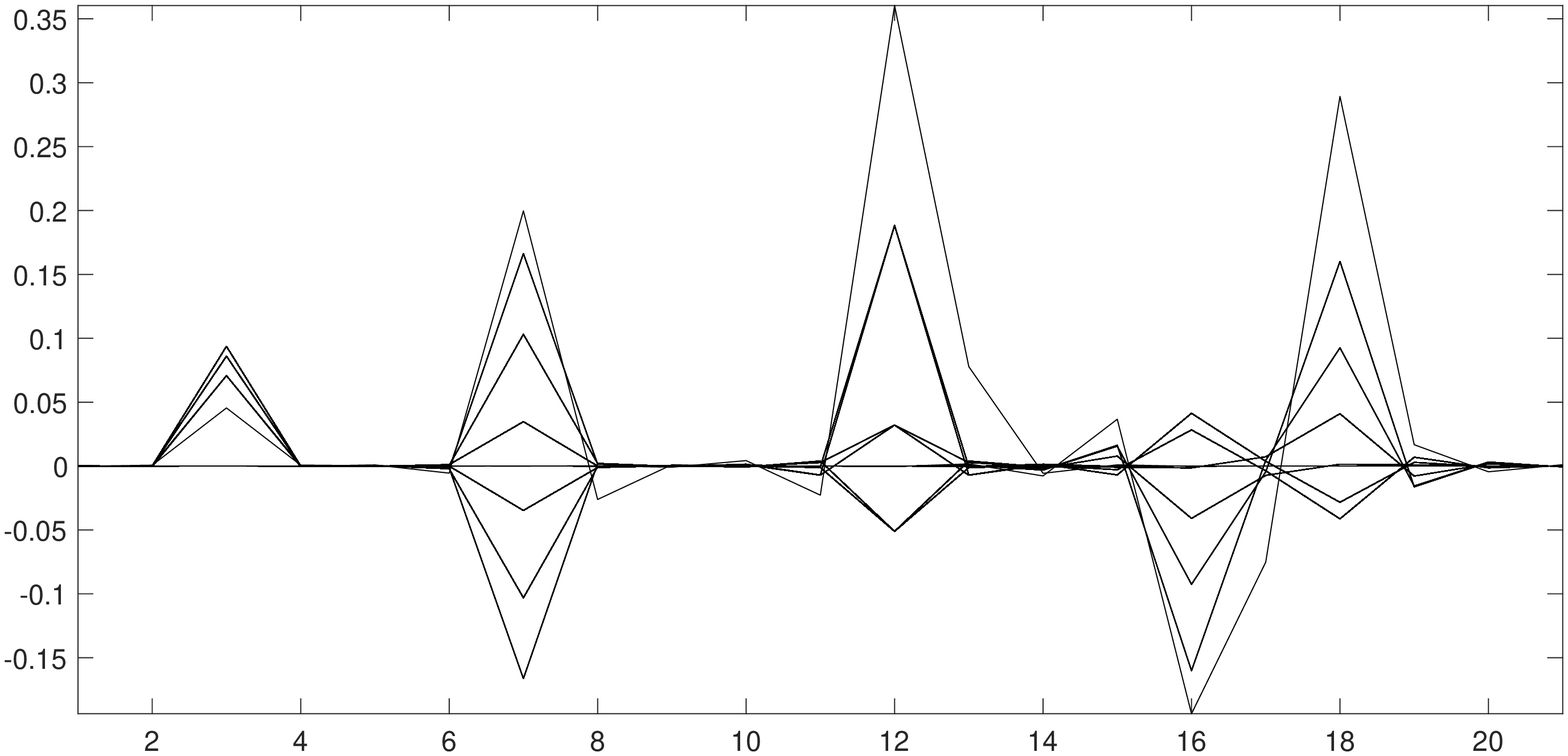}
\caption{\label{fig:pqbig_evec_128samples_61626364_b7c21_1to384}
Sample values of eigenvectors of $PQ$ on $\mathcal{B}_7\vdash \mathcal{C}_{21}$ corresponding to eigenvalue indices 61--64.
Each group plots  the set of sample values $\varphi_\nu(k+128*\ell),\ell=0,\dots 20$, $k=1,\dots, 128$ cyclically shifted by a fixed amount for each
 $\nu=61$ (left) to $\nu=64$ (right).  The samples are essentially cardinal (peak at the center and zeros elsewhere) for $\nu=61$ and become \emph{less cardinal} to the point of containing an oscillation for $\nu=64$
}
}
\end{figure}

  When $m=21$, $N=7$ and $K=3$ there are $m({\rm dim}_K-1)= 21\cdot \bigl(-1+\sum_{\kappa=0}^3 \binom{7}{\kappa}\bigr)
  =21\cdot 63=1323$
  eigenvalues of $L(\mathcal{B}_N \vdash \mathcal{C}_m)$
  with value at most $6$. Most of these come from the $21\bigl(\binom{7}{1}+\binom{7}{2}+\binom{7}{3}-3\bigr)= 21*60$ 
  $2K$-Dirichlet eigenvalues of $L(\mathcal{B}_7\vdash \mathcal{C}_{21})$ for $K=0,\dots, 3$. 
  As shown in Fig.~\ref{fig:slbig_evals_b7c21_1to620}, the zero eigenvalue and groups of 
  $m\bigl(\binom{N}{k}-1\bigr)=21\bigl(\binom{7}{k}-1\bigr)$ Dirichlet eigenvalues equal to $2k$  
  are punctuated by  transition intervals of length 21 between
  the eigenvalues $2k$ and $2(k+1)$ ($k=0,\dots, 2$).
The eigenvalues of $PQ$, where $P$ is the projection onto ${\rm PW}_6$, are plotted in 
Fig.~\ref{fig:spqbig_evals_b7c21_1to80}. There are 60 Dirichlet eigenvalues of $PQ$ equal to one, and three additional
eigenvalues close to one. Typical eigenvectors for the eigenvalues equal to one, close to one, and close to zero,
and their graph Fourier transforms, are plotted
in Fig.~\ref{fig:pqbig_evecs_and_fourier_32_62_64_b7c21_1to384}.
The eigenvalues close to one have eigenvectors whose samples $\varphi(v^k)$, where $v^k\in\mathcal{B}_N^k$, 
correspond to fixed $v\in\mathcal{B}_N$,
are nearly cardinal. These samples are plotted in 
Fig.~\ref{fig:pqbig_evec_128samples_61626364_b7c21_1to384}.

The matrix whose columns are cyclic shifts of the Neumann-type $PQ$-eigenvectors numerically has full rank. Since these shifts
are orthogonal to corresponding shifts of Dirichlet eigenvectors of $PQ$, together these shifts form a basis for ${\rm PW}_6$.
  
  
 \section{Example 2: Sampling on the Cartesian product of a cube and a cycle \label{sec:cartesian} }

\begin{figure}[tbhp]
\begin{centering}
\vspace{1cm}
\begin{tiny}
$\xymatrix{
&& (00,0)\ar@{-}[ld]  \ar@{-}[rdd]   \ar@{--}[rr] &  & (00,1)\ar@{-}[ld]  \ar@{-}[rdd]  \ar@{--}[rr]  && (00,2)\ar@{-}[ld]  \ar@{-}[rdd]   \ar@{--}@/_2pc/[llll]&&\\
& (10,0)\ar@{-}[rdd]   \ar@{--}[rr] && (10,1)\ar@{-}[rdd]   \ar@{--}[rr] &    &(10,2)\ar@{-}[rdd] \ar@{--}@/_2pc/[llll]&   && \\
 &&& (01,0)\ar@{-}[ld]  \ar@{--}[rr]    && (01,1)\ar@{-}[ld]  \ar@{--}[rr]  && (01,2)\ar@{-}[ld]  \ar@{--}@/^2pc/[llll]& \\
& &(11,0)  \ar@{--}[rr]  &&(11,1) \ar@{--}[rr]  & & (11,2)  \ar@{--}@/^2pc/[llll] && \\
}$
\end{tiny}
\vspace{0.5cm}
\caption{\label{fig:bn2quarec3} Graphic representation of $\mathcal{B}_2\square\,\mathcal{C}_3$}
\end{centering}
\end{figure}
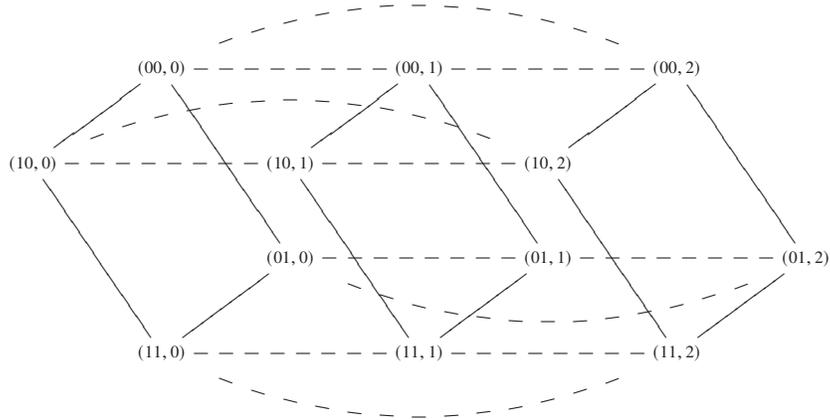

 In the Cartesian product $\mathcal{B}_N\square \,\mathcal{C}_m$ of a cube and a cycle,  each copy of the cube can be regarded as a cluster (to which we refer as a $\mathcal{B}_N$ slice). 
 In contrast to the graphs formed by substituting cubes for vertices of a cycle, in the Cartesian product with a cycle, each vertex in a single cluster is connected
 to a vertex in each of two adjacent clusters, see Fig.~\ref{fig:bn2quarec3}. This prohibits Dirichlet eigenvectors of the Laplacian 
 $L(\mathcal{B}_N\square \,\mathcal{C}_m)$. Nonetheless, $\mathcal{B}_N\square \,\mathcal{C}_m$ is $(N+2)$-regular with $N$ neighbors coming from inside a cluster, so if $N$ is large then 
 most of the degree of each vertex is associated with neighbors within the cluster.

  The eigenvalues of $\mathcal{B}_N\square \,\mathcal{C}_m$ are sums of eigenvalues of $\mathcal{B}_N$ and of 
 $\mathcal{C}_m$  \cite{brouwer_spectra}, thus have the form $2k_1+4\sin^2(\pi k_2/m)$, $0\leq k_1\leq N$, $0\leq k_2< m$. 
 Suppose that $K\geq 3$. Since the eigenvalues of $\mathcal{C}_m$ are at most four, the  $\mathcal{B}_N\square \,\mathcal{C}_m$-Laplacian eigenvalues  that are less than or equal to  $2K$
 are sums of the form (i) a $\mathcal{B}_N$-eigenvalue at most $2(K-2)$ plus any $\mathcal{C}_m$-eigenvalue; (ii) 
 a $\mathcal{B}_N$-eigenvalue at most $2(K-1)$ plus any $\mathcal{C}_m$-eigenvalue with value at most two, or 
 (iii) a $\mathcal{B}_N$-eigenvalue equal to $2K$ plus the zero eigenvalue of $\mathcal{C}_m$.  The
 $\mathcal{B}_N\square \,\mathcal{C}_m$-eigenvectors
 are tensor products of  $\mathcal{B}_N$-eigenvectors and $\mathcal{C}_m$-eigenvectors of the corresponding types, 
 whose eigenvalues add to at most $2K$. 
 Formally, letting $\Lambda_\nu(\mathcal{G})$ denote the $\lambda_\nu$-eigenspace of $L(\mathcal{G})$, we can write
 \begin{eqnarray}\label{eq:pw2k_decomp} {\rm PW}_{2K}( \mathcal{B}_N\square \,\mathcal{C}_m) 
& = & [{\rm PW}_{2K-4}( \mathcal{B}_N)\otimes \ell^2(\mathcal{C}_m) ] \notag\\
&\oplus & [\Lambda_{2K-2}( \mathcal{B}_N)\otimes {\rm PW}_2(\mathcal{C}_m) ]\notag \\
&\oplus&   [\Lambda_{2K}( \mathcal{B}_N)\otimes \Lambda_0(\mathcal{C}_m) ]\, .
 \end{eqnarray}

 The dimension of the space ${\rm PW}_{2K}( \mathcal{B}_N\square \,\mathcal{C}_m)$ 
is equal to the number of elements of a basis for ${\rm PW}_{2K}$.  
 In the next few paragraphs we count the dimensions of the three components in (\ref{eq:pw2k_decomp}), and discuss the behavior of the operator $PQ$ 
that truncates to a  $\mathcal{B}_N$ slice then bandlimits.
Specifically, let $Q$ denote the operator that truncates a vertex function $f$ to the zero slice $\mathcal{B}_N\square \{0\}$ of 
 $\mathcal{B}_N\square \,\mathcal{C}_m$. That is, $(Qf)(v)=f(v) $ if $v=(u,0)$ for some vertex $u$ of  $\mathcal{B}_N$, and
 $(Qf)(v)=0 $, otherwise.  For a fixed $K\in \{1,\dots, N-1\}$, let $P$ denote the operator that bandlimits $f$ to ${\rm PW}_{2K}$. That is, 
 $Pf=\sum_{\lambda_\nu\leq 2K}\sum_{j=1}^{{\rm dim} \Lambda_\nu} \langle f, h_{\nu, j} \rangle h_{\nu, j}$
 where $\{h_{\nu,j}\}_{j=1}^{{\rm dim} \Lambda_\nu}$ is an orthonormal 
 basis of eigenvectors  of the $\lambda_\nu$-eigenspace  $\Lambda_\nu$ of   $L(\mathcal{B}_N\square \,\mathcal{C}_m)$.
 
 Since $L(\mathcal{C}_m)$ has norm at most four, 
 ${\rm PW}_4(\mathcal{C}_m)=\ell^2(\mathcal{C}_m)$. 
In particular, the delta function $\delta(\ell)$ on $\mathcal{C}_m$  is in ${\rm PW}_4(\mathcal{C}_m)$.  It follows that  if a vertex
 function $f\in {\rm PW}_{2K-4}$ 
 is supported in the zero-$\mathcal{B}_N$ slice of $\mathcal{B}_N\square \,\mathcal{C}_m$  then  $PQf=f$.
 The space ${\rm PW}_{2K-4}(\mathcal{B}_N)$ has dimension ${\rm dim}_{K-2}=\sum_{\kappa=0}^{K-2} \binom{N}{\kappa}$. 
 Since each $\mathcal{B}_N$ slice contributes such a subspace of ${\rm PW}_{2K}(\mathcal{B}_N\square \,\mathcal{C}_m)$, 
 the component ${\rm PW}_{2K-4}( \mathcal{B}_N)\otimes \ell^2(\mathcal{C}_m)$ 
 of ${\rm PW}_{2K}(\mathcal{B}_N\square \,\mathcal{C}_m)$ in (\ref{eq:pw2k_decomp}) has  dimension 
 $m({\rm dim}_{K-2})$.
 
 Next, consider the component $\Lambda_{2(K-1)}( \mathcal{B}_N)\otimes {\rm PW}_2(\mathcal{C}_m) $  
 of ${\rm PW}_{2K}(\mathcal{B}_N\square \,\mathcal{C}_m)$.  For explicit calculations below we assume that $m$ has the form
 $4m'+1$, but the stated results hold for other values of $m$.
The factor ${\rm PW}_2(\mathcal{C}_m) $
has a basis of 
 exponentials $\{e^{2\pi i k\ell/m}\}$ such that $|k\, {\rm mod}\, m|\leq \lfloor m/4\rfloor$.  
There are $\lfloor (m+1)/2\rfloor$ such terms. Thus the space generated by corresponding tensor products has
 dimension $\lfloor (m+1)/2\rfloor \binom{N}{K-1}$.
 
Consider eigenvectors of $PQ$ that lie in this component. Those having the largest possible eigenvalue of $PQ$ are
 products of elements of $\Lambda_{2(K-1)}(\mathcal{B}_N)$ and elements of ${\rm PW}_2(\mathcal{C}_m)$ having 
 the largest possible value at zero. This value is that of the unit-normalized  Dirichlet kernel defining
 ${\rm PW}_2(\mathcal{C}_m)$ by convolution in $\ell^2(\mathcal{C}_m)$.
Since  $\{e^{2\pi i k\ell/m}\}_{\ell=0}^{m-1}$ has $L(\mathcal{C}_m)$-eigenvalue $4\sin^2(\pi k/m)$,
 the finite Dirichlet kernel 
\[D_n(\ell)=\sum_{k=-n}^n e^{2\pi i k\ell/m}=\frac{\sin (2\pi (n+1/2)\ell/m) }{\sin( \pi \ell/m)}\]
 lies in 
${\rm PW}_{2}(\mathcal{C}_m)$ provided that $n\leq \lfloor m/4\rfloor$.
The $\ell^2(\mathcal{C}_m)$-norm of $D_n$ is $m(2n+1)$, so $\bar{D}_n=D_n/\sqrt{m(2n+1)}$ has unit norm and $\bar{D}_n(0)=\sqrt{(2n+1)/m}$.
When $n=m'=(m-1)/4$, one has $\bar{D}_n\in {\rm PW}_2(\mathcal{C}_m)$ and 
$\bar{D}_n(0)=\sqrt{1/2+1/(2m)}$.
If $f(u,\ell)=\psi(u) \bar{D}_{m'}(\ell)$ is in the $(2K-2)$-eigenspace of the $\mathcal{B}_N$-Laplacian on  
$\mathcal{B}_N\square \{0\}$ then
$f\in {\rm PW}_{2K}$ and 
$PQ (\bar{D}_{m'} \psi)=P_{2,\mathcal{C}_m}(\delta_{0,\ell}\bar{D}_{m'}(\ell) )P_{2K-2,\mathcal{B}_N} \psi(u)=\bar{D}_{m'}(0)\bar{D}_{m'}(\ell) \psi(u)=\bar{D}_{m'}(0) f $ 
where $\bar{D}_{m'}(0)=\sqrt{\frac{1}{2}+\frac{1}{2m}}$ and  $P_{\lambda,\mathcal{G}}$ denotes the projection of $\ell^2(\mathcal{G})$ onto ${\rm PW}_\lambda(\mathcal{G})$.  
The space of such $f$ has dimension $\binom{N}{K-1}$.  
The  subspace of $\ell^2(\mathcal{C}_m)$ spanned by the cyclic shifts $D_n(\cdot-\ell)$
has dimension $2n+1$ when $n\leq \lfloor (m-1)/2\rfloor $. To see this,
observe that the Fourier transform
of $D_n(\cdot-k)$ is $e^{-2\pi i \xi k/m} \widehat{D}_n(\xi) 
= e^{-2\pi i \xi k/m}\mathbbm{1}_{|\xi {\rm mod}\, m|\leq n}(\xi)$.  Since $2n+1\leq m$, the span of the truncated exponentials
$\{e^{-2\pi i \xi k/m}\mathbbm{1}_{|\xi  {\rm mod}\ m|\leq n}(\xi)\}_k$ has dimension at most $2n+1$ and so the same holds
for the shifts $D_n(\cdot-k)$. 
In the case $n=m'=(m-1)/4$, 
$2m'+1 
=\lfloor (m+1)/2\rfloor$. 
 The space generated by the cyclic shifts  $\{D_{m'}(\cdot-\ell) \psi(u):\, \psi\in\Lambda_{2(K-1)}(\mathcal{B}_N)\}$ then
has dimension $\lfloor (m+1)/2\rfloor\, \binom{N}{K-1}$.

Finally, if the restriction of $f$ to $\mathcal{B}_N\square \{0\}$ 
 is in the $2K$-eigenspace $\Lambda_{2K}$ of the $\mathcal{B}_N$-Laplacian then the only way in which
 $f$ itself can lie in ${\rm PW}_{2K}$ is if $f$ is the tensor product of its zero slice with the constant function equal to one on 
 $\mathcal{C}_m$:  $f(u,\ell)=c \mathbbm{1}_{\mathcal{C}_m} (\ell) \psi(u)$.
 The space of such functions has dimension $\binom{N}{K}$. 
 In this case, $PQ f=P(\delta_{0,\ell}f)=c P_{0,\mathcal{C}_m}(\delta_{0})P_{2K,\mathcal{B}_N} \psi=c\frac{1}{\sqrt{m}}\mathbbm{1}_{\mathcal{C}_m} (\ell) \psi(u)=\frac{1}{\sqrt{m}} f$.

 According to (\ref{eq:pw2k_decomp}), combining  bases for ${\rm PW}_{2(K-2)}(\mathcal{B}_N)\otimes\ell^2(\mathcal{C}_m)$, 
 for $\Lambda_{2K-2}(\mathcal{B}_N)\otimes {\rm PW}_{2}(\mathcal{C}_m)$, and  for 
  $\Lambda_{2K}(\mathcal{B}_N)\otimes {\rm PW}_{0}(\mathcal{C}_m)$,
produces
 a basis for ${\rm PW}_{2K}$. 
Its dimension is thus $m (\sum_{k=0}^{K-2} )+\lfloor (m+1)/2\rfloor \binom{N}{K-1}+\binom{N}{K}$.
 We collect the observations above. 
 
 \begin{theorem}\label{thm:bnsquarecm}
 For $0<K<N$ the space ${\rm PW}_{2K}( \mathcal{B}_N\square \,\mathcal{C}_m) $ is the orthogonal direct sum of
 the spaces (i) ${\rm PW}_{2K-4}( \mathcal{B}_N)\otimes \ell^2(\mathcal{C}_m) $, (ii) 
$\Lambda_{2K-2}( \mathcal{B}_N)\otimes {\rm PW}_2(\mathcal{C}_m)$, and (iii)
 $\Lambda_{2K}( \mathcal{B}_N)\otimes \Lambda_0(\mathcal{C}_m) $. 
  One has the following:
 \begin{enumerate}
 \item The space ${\rm PW}_{2K-4}( \mathcal{B}_N)\otimes \ell^2(\mathcal{C}_m) $ has dimension $m\sum_{\kappa=0}^{K-2}\binom{N}{\kappa}$.  An orthonormal basis for this space consists of cyclic shifts of the vectors $\delta_{0} (\ell) \psi_{\kappa,\nu}(u)$ where $\{\psi_{\kappa,\nu}\}_\nu$ is an ONB for $\Lambda_{2\kappa}(\mathcal{B}_N)$. On the zero-$\mathcal{B}_N$ slice, for any $f$
 in this space one has
 \[ \|Qf\|^2=\sum_{\kappa=0}^{K-2}\sum_{\nu=1}^{{\rm dim}\,\Lambda_\kappa} |\langle f,\, \psi_{\kappa,\nu}\rangle|^2\, .
 \]
 \item The space $\Lambda_{2K-2}( \mathcal{B}_N)\otimes {\rm PW}_2(\mathcal{C}_m)$ has dimension $\lfloor (m+1)/2\rfloor \binom{N}{K-1}$ when $m$ has the form $m=4m'+1$.  A basis of this space consists of the cyclic shifts 
 $\bar{D}_{m'}(\cdot-2\ell)\, \psi_{K-1,\nu}$, $\ell=0,\dots, (m-1)/2$, $\nu=1,\dots, {\rm dim}\, \Lambda_{2(K-1)}$.
On the zero-$\mathcal{B}_N$ slice, for any $f$
 in this space one has
 \[ \Bigl(\frac{1}{2}+\frac{1}{2m}\Bigr)\|Q f\|^2=
 \sum_{\nu=1}^{{\rm dim}\, \Lambda_{2K-2}} |\langle f,\, \psi_{K-1,\nu}\rangle|^2\, .
 \]
 \item 
 The space $\Lambda_{2K}( \mathcal{B}_N)\otimes {\rm PW}_0(\mathcal{C}_m)$ has dimension $\binom{N}{K}$.  
 An orthonormal basis  of this space consists of
 $\frac{1}{\sqrt{m}}\mathbbm{1}_{\mathcal{C}_m}(\ell) \psi_{K,\nu}(u)$, $\nu=1,\dots, \binom{N}{K}$ where $\{\psi_{K,\nu}\}$ is
 an ONB for $\Lambda_{2K}(\mathcal{B}_N)$.
On the zero-$\mathcal{B}_N$ slice, for any $f$
 in this space one has
 \[ \frac{1}{m}\|Q f\|^2= \sum_{\nu=1}^{{\rm dim}\, \Lambda_{2K}} |\langle  f,\, \psi_{K,\nu}\rangle|^2\, .
\]
 \end{enumerate}
 \end{theorem}
In contrast to the reconstruction functions $p_\nu$ in (\ref{eq:gensamp}), the functions $Q\psi_{k,\nu}$ are not themselves in ${\rm PW}_{2K}$
when $k=K-1$ or $k=K$. However, the $Q\psi_{k,\nu}$ are mutually orthogonal on $\mathcal{B}_N\,\square\,\{0\}$ and have unique extensions 
$\psi_{k,\nu}\in {\rm PW}_{2K}$.

In comparison with Thm.~\ref{thm:pesenson2}, Thm.~\ref{thm:bnsquarecm} identifies three separate regimes 
(localized on a slice, concentrated on a slice, and equally distributed over all slices)--in a very specific case, in which
a graph $\mathcal{G}$ might satisfy a Plancherel--Polya type inequality on each cluster, in each regime, but the 
corresponding constants ($1$, $1/2+1/(2m)$ and $1/m$ respectively  reflect the nature of the regime.

\subsection{Specific case $N =7$, $K =3$, $m=21$ of $\mathcal{B}_N\square \,\mathcal{C}_m$ \label{ssec:N7m21_cartesian}}
 
As we did in the case of $\mathcal{B}_N\vdash \mathcal{C}_m$, we fix $N=7$ and $K=3$. The dimension of ${\rm PW}_6(\mathcal{B}_3\square\,\, \mathcal{C}_{21})$ is 
\[21\biggl(\binom{3}{0}+\binom{3}{1}\biggr)+\lfloor (21+1)/2\rfloor \binom{7}{2}+\binom{7}{3}
=21\cdot 8+11\cdot 21 +35 =434 .\]
The corresponding Laplacian eigenvalues that are at most six are shown in Fig.~\ref{fig:sc21_evals_b7c21_1to620}.
The 64 nonzero eigenvalues of $PQ$ are plotted in Fig.~\ref{fig:spqc21_evals_b7c21_1to80}.
 
 Consider, on the other hand, the eigenvalues of $PQ$ as described above.  In the case $N=7, K=3$ and $m=21$ there are 
 8 eigenvalues of $PQ$ equal to one, 21 equal to $1/2+1/42$, and 35 equal to $1/21$. The remaining eigenvalues are equal to zero.

     \begin{figure}[tbhp]
  \footnotesize{
\centering 
\includegraphics[width=\textwidth,height=2.5in]{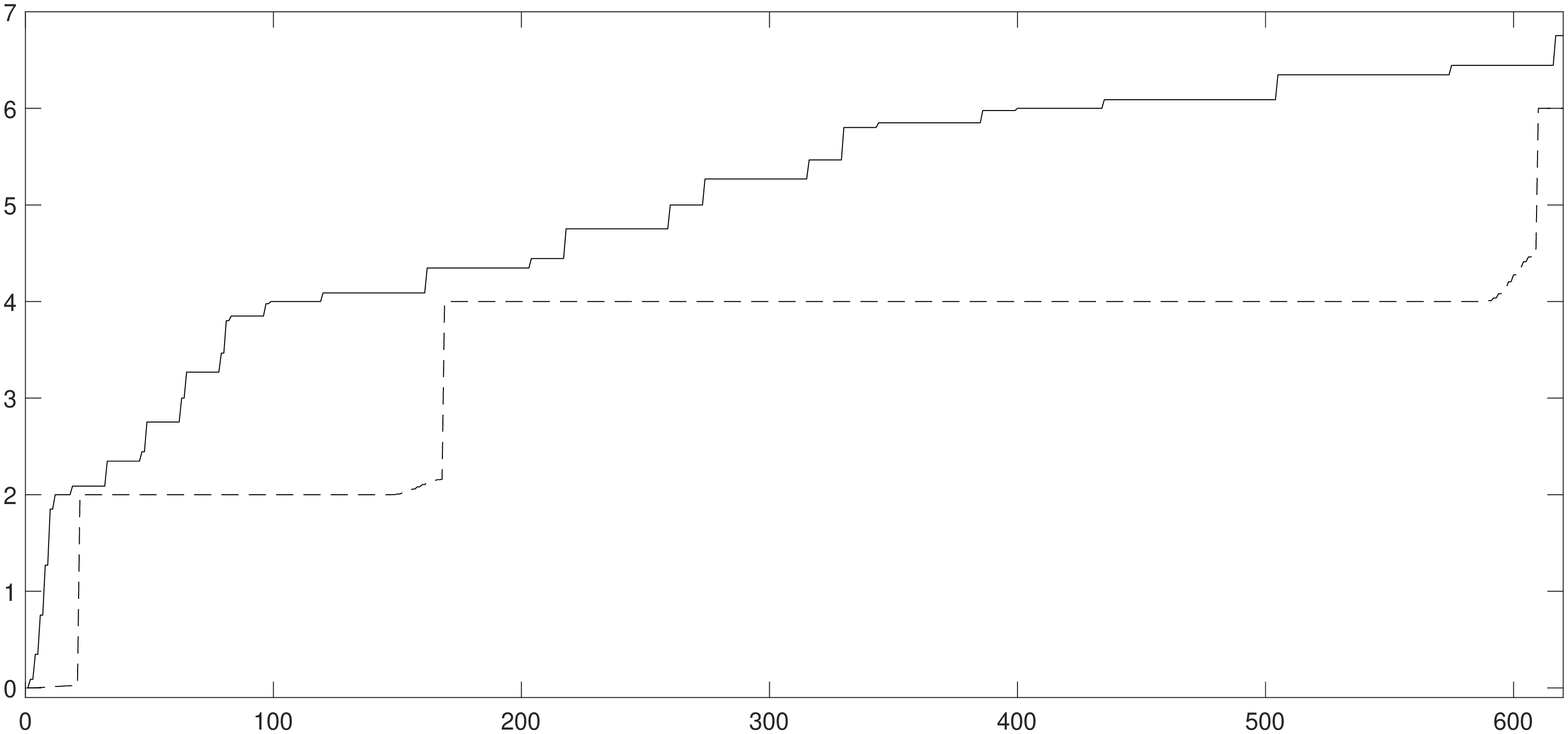}
\caption{\label{fig:sc21_evals_b7c21_1to620}
First 620 eigenvalues of $L(\mathcal{B}_7\square \mathcal{C}_{21})$. Eigenvalues of $L(\mathcal{B}_7\vdash \mathcal{C}_{21})$
are plotted (dashed) for comparison
}
}
\end{figure}

    \begin{figure}[tbhp]
  \footnotesize{
\centering 
\includegraphics[width=\textwidth,height=2.5in]{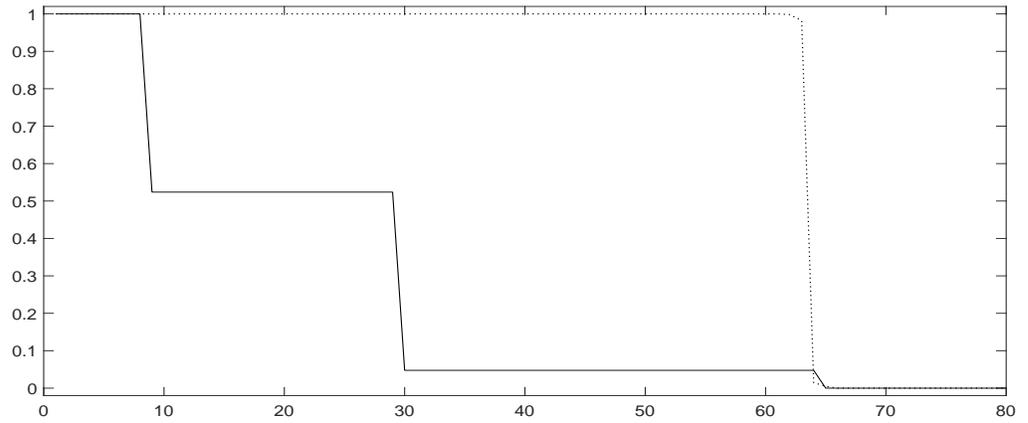}
\caption{\label{fig:spqc21_evals_b7c21_1to80}
Eigenvalues of $PQ$ on $\mathcal{B}_7\square \mathcal{C}_{21}$. 
Corresponding eigenvalues of $PQ$ on $\mathcal{B}_7\vdash \mathcal{C}_{21}$ are plotted (dashed) for comparison
}
}
\end{figure}

    \begin{figure}[tbhp]
  \footnotesize{
\centering 
\includegraphics[width=\textwidth,height=2.5in]{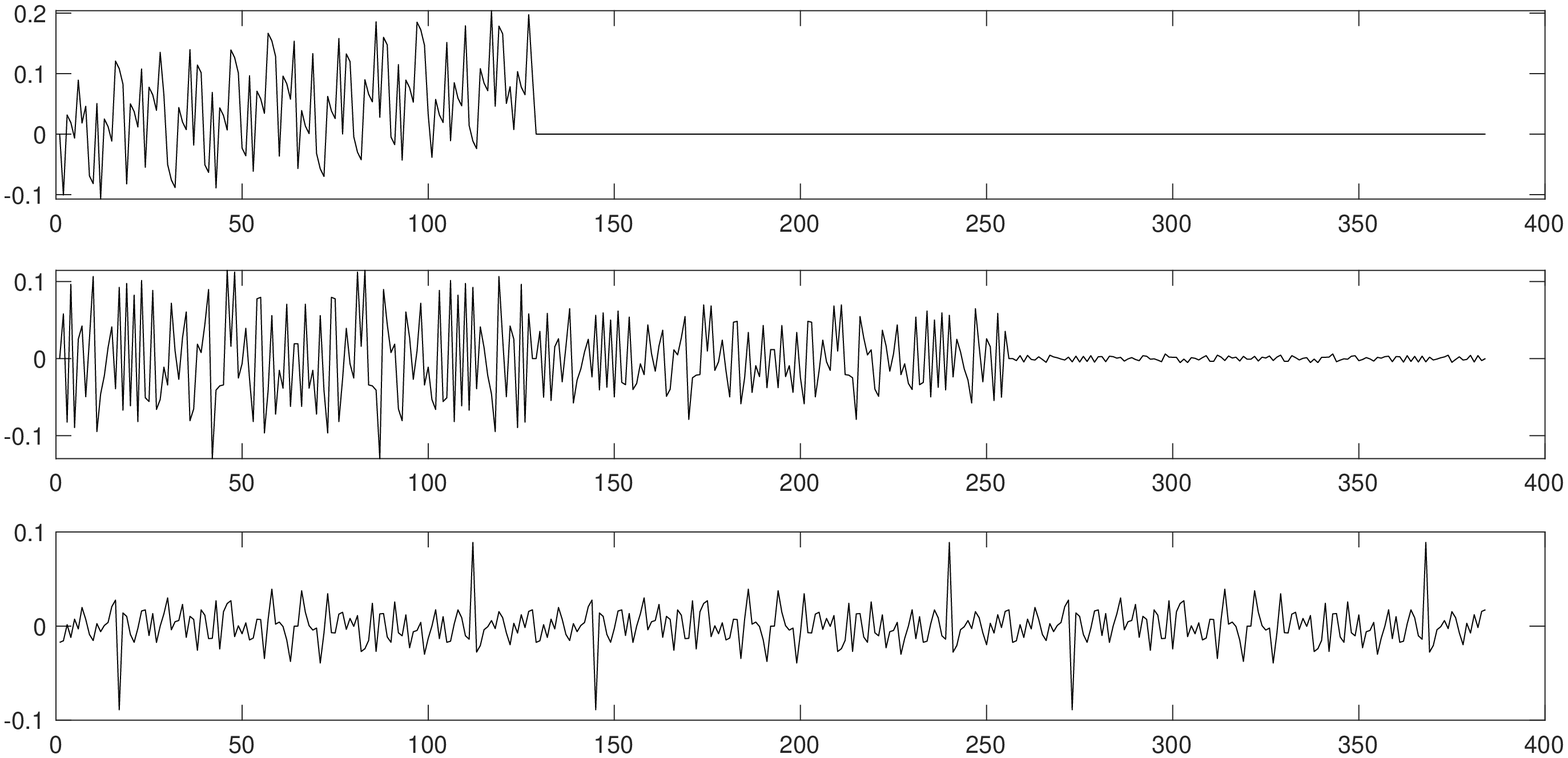}
\caption{\label{fig:pqc21_evecs1_11_32_b7c21_1to384}
Eigenvectors of $PQ$ on $\mathcal{B}_7\square \mathcal{C}_{21}$ corresponding to eigenvalue indices 1 (top), 11 and 30 (bottom)
whose respective eigenvalues are $1$, $1/2+1/42$ and $1/21$. The first three block lengths ($3*128$) are plotted. The bottom eigenvector is periodic of period 128.
}
}
\end{figure}

\section{Discussion and conclusions\label{sec:conclusion}}
We have considered two simple families of graphs, $\mathcal{B}_N\vdash \mathcal{C}_m$ and $\mathcal{B}_N\,\square\, \mathcal{C}_m$ 
in which cubes $\mathcal{B}_N$ are thought of as clusters connected by a cycle.  A very rough measure
of \emph{clusterness} of a subgraph determined by a subset of vertices, most of whose neighbors are in the subset, is the ratio of the average number of neighbors 
within the cluster to the average number of neighbors in the whole graph. In the case of $\mathcal{B}_N\square \,\mathcal{C}_m$
this ratio is $N/(N+2)$ whereas in the case of $\mathcal{B}_N\vdash\mathcal{C}_m$ the ratio is 
$(N2^N/(N2^N+2)=N/(N+2^{1-N})$, which is much closer to one.  

In the case of $\mathcal{B}_N\vdash\mathcal{C}_m$,  a majority of Laplacian eigenvectors  are completely supported
in a cluster. This allows for effective sampling inequalities in Paley--Wiener spaces in which samples are inner products
with eigenvectors of spatio--spectral limiting $PQ$-operators.

In the case of  $\mathcal{B}_N\square \,\mathcal{C}_m$ there are three regimes: one in which eigenvectors of $PQ$ are completely
localized in clusters, a second in which more than half of the energy of eigenvectors of $PQ$ is concentrated in a cluster, and one in which
the corresponding eigenvectors of $PQ$ are equally spread over the full graph.  (One could avoid this third regime by limiting the spectrum
to Laplacian eigenvalues strictly smaller than $2K$ for some $K$.) The measurement vectors in Thm.~\ref{thm:bnsquarecm}, which have the form
$Q\psi$ for eigenvectors of $PQ$, thus are not all themselves in the Paley--Wiener space as they are in \cite{pesenson_2021},
 but as localizations of eigenvectors of $PQ$, they are mutually orthogonal on the cluster.

We acknowledge that $\mathcal{B}_N\vdash\mathcal{C}_m$ and $\mathcal{B}_N\square \,\mathcal{C}_m$ are far from being graphs that occur
in real networks. However, the three regimes governing concentration of elements of Paley--Wiener spaces on clusters 
in the example of  $\mathcal{B}_N\square \,\mathcal{C}_m$ should be a feature in
graphs whose clusters have induced Laplacian eigenvalues that are separated in magnitude on the order of the norm of the \emph{skeleton} graph 
in which each cluster is compressed down to a single vertex.

%
%
%
%
%


 
 \bigskip\section*{Appendix: $PQ$ on a finite abelian group and spectral accumulation \label{sec:finite_abelian}}
 \addcontentsline{toc}{section}{Appendix}
 In Sect.~\ref{ssec:gensamp} it was observed that  generalized sampling expansions involve
 expansion of the exponential Fourier kernel $e^{2\pi i t\xi}$ in terms of inversion of a matrix with entries indexed
 by Fourier transforms of sampling convolvers (\ref{eq:exp_gensamp}).  The particular case when these convolvers
 are eigenfunctions of time and band limiting that we outlined
 in Sect.~\ref{ssec:tbl}  has an analogous version in the more general setting of locally compact abelian groups, and we 
 outline the case of finite abelian groups here.
 
Let $G$ be a finite abelian group.  By the fundamental theorem of finite abelian groups,
 $G$ is a product $\prod_{\nu=1}^N \mathbb{Z}_{m_\nu}$ of cycles of lengths that divide the order $|G|$ of $G$.  A Fourier basis consists of tensor products of
 normalized exponential vectors $\{e^{2\pi i \ell k_\nu/m_\nu}/\sqrt{m_\nu} \}_{\ell=0}^{m_\nu-1}$ indexed by $k_\nu=0,\dots, m_\nu-1$. The Fourier transform elements form a dual group $\widehat{G}$ under componentwise multiplication.  $\widehat{G}$ and $G$ are isomorphic .
  We make use of these facts to establish a \emph{spectral accumulation property} of spatio--spectral localization operators on $G$ 
  (see \cite{hogan_lakey_book} for the case on $\mathbb{R}$, which is a well-known).
  In what follow we assume indexings of the vertex elements $s\in G$ and  $\sigma\in \widehat{G}$.  We denote by 
  $F(s,\sigma)$ the matrix of the Fourier transform of $G$ with respect to these indexings such that the columns (fixed $\sigma$)
  are pairwise orthonormal: $FF^\ast=I_{|G|}$.  Since the entries are products of normalized exponentials,  
  $|F(s,\sigma)|^2=1/|G|$.
  
Suppose that $S$ and $\Sigma$ are symmetric subsets of $G, \widehat{G}$ (e.g., $S=-S$).  Let $P$ be the orthogonal projection onto
the span of the vectors $\{F(\cdot,\sigma):\,\sigma\in\Sigma\}$ and let $Q$ be pointwise multiplication by the characteristic
function of $S$. Then $PQ$ is self-adjoint and has rank equal to $|\Sigma|$ if $|S|\geq |\Sigma|$. 
Let $\mu_1\geq \mu_2\geq \cdots\geq \mu_{|\Sigma|}$ 
be an ordering of the nonzero eigenvalues  $PQ$ with corresponding eigenfunctions $\varphi_\nu$ so that (by Mercer's theorem) $K(s,t)=\sum_{\nu=1}^{|\Sigma|} 
\varphi_\nu(s) \overline{\varphi_\nu(t)}$ is the kernel of $PQ$. For each $\sigma\in \Sigma$
one has 
\[F(s,\sigma)=\sum_\nu  \langle F(\cdot,\sigma),\,\varphi_\nu\rangle \, \varphi_\nu(s)\, 
=\sum_\nu \frac{1}{\mu_\nu} \langle F(\cdot,\sigma),\, Q \varphi_\nu\rangle \, \varphi_\nu(s)\, 
\]
since $PQ\varphi_\nu=\mu_\nu \varphi_\nu$ and $F(\cdot,\sigma)  PQ\varphi_\nu=F(\cdot,\sigma)  Q\varphi_\nu$ for $\sigma\in\Sigma$.

Therefore,
\begin{eqnarray*}\sum \mu_\nu |F\varphi_\nu(\sigma)|^2 &=& \sum \frac{1}{\mu_\nu}|F(PQ\varphi_\nu)(\sigma)|^2
=\sum \frac{1}{\mu_\nu} \mathbbm{1}_\Sigma|F(Q\varphi_\nu)(\sigma)|^2\\
&=&\sum \frac{1}{\mu_\nu} \mathbbm{1}_\Sigma F(Q\varphi_\nu)(\sigma) \overline{F(Q\varphi_\nu)(\sigma)}\\
&=&\sum \frac{1}{\mu_\nu} \mathbbm{1}_\Sigma \langle F(\cdot, \sigma), (Q\varphi_\nu)(\cdot)
\rangle  \overline{\langle F(\cdot, \sigma), (Q\varphi_\nu)(\cdot)
\rangle}\\
&=&\mathbbm{1}_\Sigma \bigl\langle  \sum_\nu \frac{1}{\mu_\nu} \langle F(\cdot,\sigma), Q\varphi_\nu)\rangle \,\varphi_\nu,\, Q F(\cdot,\sigma) \bigr\rangle  \\
&=&\mathbbm{1}_\Sigma \bigl\langle  F(\cdot,\sigma), \,  \mathbbm{1}_S F(\cdot,\sigma) \bigr\rangle  
=\mathbbm{1}_\Sigma(\sigma)   \sum_{s\in S} F(s,\sigma) \overline{F(s,\sigma)}\\
& &=\frac{1}{|G|}\mathbbm{1}_\Sigma(\sigma) \sum \mathbbm{1}_S =\frac{1}{|G|} \mathbbm{1}_\Sigma(\sigma) |S|\\
\end{eqnarray*}
using that $Q$ is self adjoint. The group structure played a critical role in that the kernel of the Fourier matrix is made of elements
of modulus equal to one.

We summarize the calculation above in the following.
\begin{proposition}\label{prop:finite_abelian} Let $G$ be a finite abelian group with dual $\widehat{G}$ and let $S$ and $\Sigma$ be symmetric subsets of $G$ and $\widehat{G}$ respectively with $|S|\geq |\Sigma|$. Denote by $P$ the orthogonal projection onto the span of the Fourier vectors indexed by $\Sigma$ and let $(Qf)(s)=\mathbbm{1}_S(s) f(s)$ such that $\mu_1\geq\cdots\geq \mu_{|\Sigma|}$
are the eigenvalues of $PQ$ with eigenvectors $\varphi_1,\dots,\varphi_{|\Sigma|}$.
\[\sum \mu_\nu |F\varphi_\nu(\sigma)|^2=\frac{1}{|G|} \mathbbm{1}_\Sigma(\sigma) |S|
\]
\end{proposition}


 \bibliographystyle{amsplain}
\bibliography{higgins_refs.bib}

\providecommand{\bysame}{\leavevmode\hbox to3em{\hrulefill}\thinspace}
\providecommand{\MR}{\relax\ifhmode\unskip\space\fi MR }
\providecommand{\MRhref}[2]{%
  \href{http://www.ams.org/mathscinet-getitem?mr=#1}{#2}
}
\providecommand{\href}[2]{#2}
\begin{thebibliography}{10}

\bibitem{brouwer_spectra}
A.E. Brouwer and W.H. Haemers, \emph{Spectra of graphs}, Universitext,
  Springer, New York, 2012.

\bibitem{brown_1981}
J.~Brown, \emph{Multi-channel sampling of low-pass signals}, IEEE Transactions
  on Circuits and Systems \textbf{28} (1981), no.~2, 101--106.

\bibitem{higgins_cardinal_1985}
J.R. Higgins, \emph{Five short stories about the cardinal series}, Bull. Amer.
  Math. Soc. (N.S.) \textbf{12} (1985), 45--89.

\bibitem{hogan_lakey_tbl}
J.A. Hogan and J.~Lakey, \emph{{D}uration and {B}andwidth {L}imiting. {P}rolate
  {F}unctions, {S}ampling, and {A}pplications.}, {B}irkh{\"a}user, {B}oston,
  {M}{A}, 2012.

\bibitem{hogan_lakey_2015}
J.A. Hogan and J.~Lakey, \emph{Frame properties of shifts of prolate spheroidal
  wave functions}, Applied and Computational Harmonic Analysis \textbf{39}
  (2015), 21--32.

\bibitem{Hogan2018}
J.A. Hogan and J.~Lakey, \emph{An analogue of {S}lepian vectors on {B}oolean
  hypercubes}, J. Fourier Anal. Appl. \textbf{25} (2019), no.~4, 2004--2020.

\bibitem{hogan2018spatiospectral}
\bysame, \emph{Spatio-spectral limiting on {B}oolean cubes}, J. Fourier Anal.
  Appl. \textbf{27} (2021), 40.

\bibitem{hogan_lakey_book}
J.A. Hogan and J.D. Lakey, \emph{Time--{F}requency and {T}ime--{S}cale
  {M}ethods}, Birkh\"auser Boston Inc., Boston, MA, 2005.

\bibitem{landau_widom}
H.J. Landau and H.~Widom, \emph{Eigenvalue distribution of time and frequency
  limiting}, J. Math. Anal. Appl. \textbf{77} (1980), 469--481.

\bibitem{papoulis_1968}
A.~Papoulis, \emph{Systems and transforms with application in optics},
  McGraw-Hill, New York, 1968.

\bibitem{papoulis-gensamp}
A.~Papoulis, \emph{Generalized sampling expansion}, IEEE Trans. Circuits and
  Systems \textbf{24} (1977), 652--654.

\bibitem{pesenson_2021}
I.Z. Pesenson and M.Z. Pesenson, \emph{Graph signal sampling and interpolation
  based on clusters and averages}, J. Fourier Anal. Appl. \textbf{27} (2021),
  no.~3, Paper No. 39, 28. \MR{4248652}

\bibitem{strichartz_2010}
R.S. Strichartz, \emph{Transformation of spectra of graph {L}aplacians}, Rocky
  Mountain J. Math. \textbf{40} (2010), no.~6, 2037--2062. \MR{2764237}

\bibitem{strichartz_2016}
\bysame, \emph{Half sampling on bipartite graphs}, J. Fourier Anal. Appl.
  \textbf{22} (2016), no.~5, 1157--1173. \MR{3547716}

\bibitem{tsitsvero_etal_2015}
M.~{Tsitsvero}, S.~{Barbarossa}, and P.~{Di Lorenzo}, \emph{Signals on graphs:
  Uncertainty principle and sampling}, IEEE Trans. Signal Process. \textbf{64}
  (2016), 4845--4860.

\end{thebibliography}

\end{document}